\documentclass[11pt, a4paper]{article}

\usepackage[latin1]  {inputenc}
\usepackage[ngerman,english] {babel}
\usepackage[T1]      {fontenc}
\usepackage{lmodern}
\usepackage{hyperref}

\usepackage{dbnsymb}
\usepackage{amsmath}
\usepackage{amssymb}
\usepackage{amsfonts}
\usepackage{amsxtra}
\usepackage{amsthm}
\usepackage{MnSymbol}
\usepackage[pdftex]{graphicx}
\usepackage{graphicx,verbatim,longtable,
  epic,eepic,stmaryrd,rotating}
\usepackage{multirow}

\usepackage[all, knot]{xy}

\theoremstyle{plain}
\newtheorem{satz}{Proposition}[section]
\newtheorem{theo}[satz]{Theorem}

\newtheorem{koro}[satz]{Corollary}
\newtheorem{lemm}[satz]{Lemma}

\theoremstyle{definition}
\newtheorem{defin}[satz]{Definition}

\theoremstyle{remark}
\newtheorem{beisp}[satz]{Example}
\newtheorem{bem}  [satz]{Remark}

\newcounter{countmatrix}
\newcommand{\nat}{\mathbb N}
\newcommand{\zet}{\mathbb Z}
\newcommand{\rat}{\mathbb Q}
\newcommand{\real}{\mathbb R}

\DeclareMathOperator{\Id}{id}

\newcommand{\iso}{\simeq}

\DeclareMathOperator{\interior}{Int}

\DeclareMathOperator{\Ob}{Ob}
\DeclareMathOperator{\Mor}{Mor}

\DeclareMathOperator{\RMOD}{\textbf{R-MOD}}

\DeclareMathOperator{\KAR}{\textbf{Kar}}

\DeclareMathOperator{\ucob}{\textbf{uCob}^2}

\DeclareMathOperator{\mat}{\textbf{Mat}}
\DeclareMathOperator{\kom}{\textbf{Kom}}

\DeclareMathOperator{\kob}{\textbf{Kob}}
\DeclareMathOperator{\ukob}{\textbf{uKob}}

\newcommand{\bn}[1]{\llbracket #1 \rrbracket}
\newcommand{\uu}{{\Large$\rcurvearrowup\lcurvearrowup$}}
\newcommand{\dd}{{\Large$\lcurvearrowdown\rcurvearrowdown$}}

\newcommand{\du}{{\Large$\lcurvearrowdown\lcurvearrowup$}}

\newcommand{\rir}{{\Large\raisebox{-0.1cm}{$\lcurvearrowright$\makebox[0pt][r]{\raisebox{0.2cm}{$\rcurvearrowright$}}}}}

\newcommand{\ril}{{\Large\raisebox{-0.1cm}{$\lcurvearrowright$\makebox[0pt][r]{\raisebox{0.2cm}{$\lcurvearrowleft$}}}}}

\newcommand{\uup}{\raisebox{-0.1cm}{\includegraphics[scale=0.5]{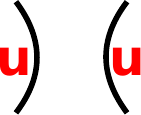}}}
\newcommand{\uupp}{\raisebox{-0.215cm}{\includegraphics[scale=0.5]{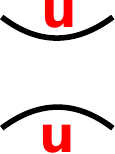}}}
\newcommand{\udp}{\raisebox{-0.1cm}{\includegraphics[scale=0.5]{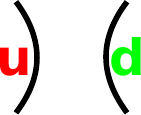}}}
\newcommand{\udpp}{\raisebox{-0.215cm}{\includegraphics[scale=0.5]{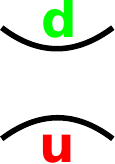}}}
\newcommand{\dup}{\raisebox{-0.1cm}{\includegraphics[scale=0.5]{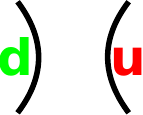}}}
\newcommand{\dupp}{\raisebox{-0.215cm}{\includegraphics[scale=0.5]{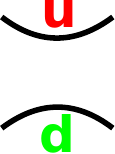}}}
\newcommand{\ddp}{\raisebox{-0.1cm}{\includegraphics[scale=0.5]{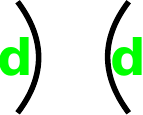}}}
\newcommand{\ddpp}{\raisebox{-0.215cm}{\includegraphics[scale=0.5]{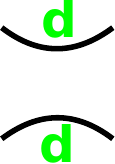}}}
\newcommand{\smo}{\raisebox{-0.1cm}{\includegraphics[scale=0.5]{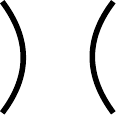}}}
\newcommand{\hsmo}{\raisebox{-0.12cm}{\includegraphics[scale=0.5]{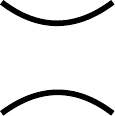}}}
\newcommand{\none}{\raisebox{-0.1cm}{\includegraphics[scale=0.5]{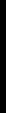}}}
\newcommand{\down}{\raisebox{-0.1cm}{\includegraphics[scale=0.5]{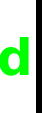}}}
\newcommand{\up}{\raisebox{-0.1cm}{\includegraphics[scale=0.5]{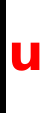}}}

\newlength{\standardunitlength}
\newcommand{\jpg}[2]{{\hspace{-3pt}\begin{array}{c}%
  \raisebox{-2.5pt}{\includegraphics[width=#1]{figs/#2.pdf}}%
\end{array}\hspace{-3pt}}}

\newcounter{In}

\title{Khovanov homology for virtual tangles and applications}
\author{Daniel Tubbenhauer\thanks{The author was supported by the German Research 
Foundation (Deutsche Forschungsgemeinschaft (DFG)) 
through the Institutional Strategy of the University of G\"{o}ttingen and by the Graduiertenkolleg 1493 ``Mathematische Strukturen in der modernen Quantenphysik''.}}
\date{Last compiled \today; last edited \today}

\begin{document}

\maketitle
\begin{abstract}
We extend the cobordism based categorification of the virtual Jones polynomial to virtual tangles. This extension is combinatorial and has semi-local properties. We use the semi-local property to prove an applications, i.e. we give a discussion of Lee's degeneration of virtual homology.
\end{abstract}
\setcounter{tocdepth}{2}
\tableofcontents
\begin{figure}[p]
\centering
     \includegraphics[scale=0.68]{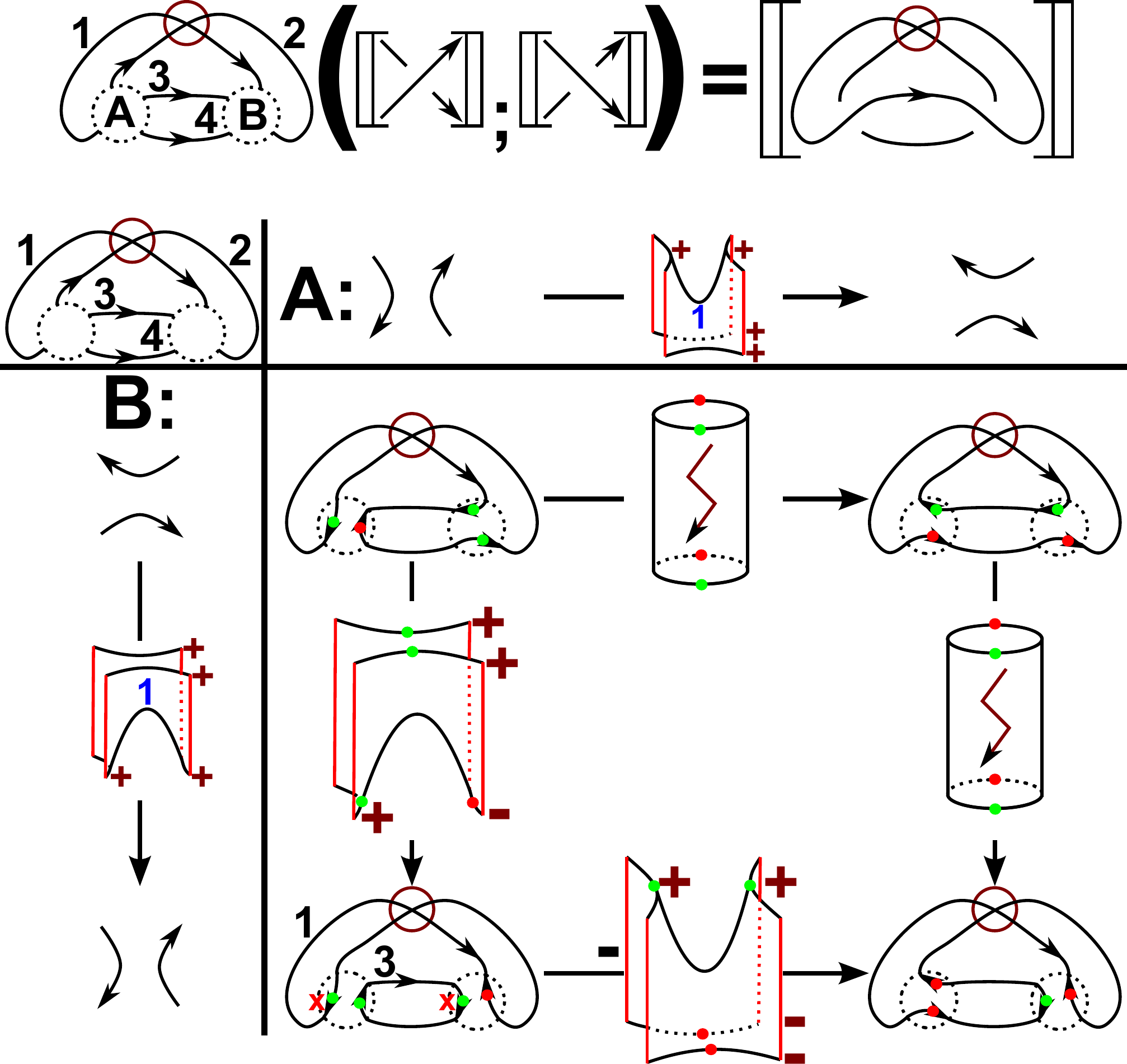}
  \caption{The main example. After we have fixed an orientation/numbering of the circuit diagram we only have to compare whether the local orientations match (green) or mismatch (red) and compose if necessary with $\Phi^-_+$ (red). Iff we have a double mismatch at the top and bottom, then we add a bolt symbol.}
  \label{figure0-main}
\end{figure}
\newpage
\section*{Introduction}\label{intro}
\addcontentsline{toc}{section}{Introduction}
Virtual knots and virtual links are a generalisation of classical knots and classical links. They were introduced by L.~Kauffman. They can be seen as a combinatorial interpretation of copies of $S^1\times\cdots\times S^1$ embedded in thickened surfaces with genus $g$ instead of copies of $S^1\times\cdots\times S^1$ embedded in $S^3$.

Virtual link diagrams are four valent graph embedded into $\mathbb{R}^2$ together with a decoration of every vertex. The decorations are either \textit{overcrossings} $\slashoverback$, \textit{undercrossings} $\backoverslash$ or \textit{virtual crossings}, latter are marked with a circle.

A virtual link is an equivalence class of virtual link diagrams modulo the \textit{generalised Reidemeister moves} in Figure \ref{figure0-1} and isotopies.

The notion of \textit{oriented} virtual link diagrams and an \textit{oriented} virtual links are equivalent, but latter modulo \textit{oriented} generalised Reidemeister moves and isotopies.

We call the moves RM1, RM2 and RM3 the \emph{classical moves}, the moves vRM1, vRM2 and vRM3 the \emph{virtual moves} and the move mRM the \emph{mixed move}. 

\begin{figure}[ht]
  \centering
     \includegraphics[scale=0.65]{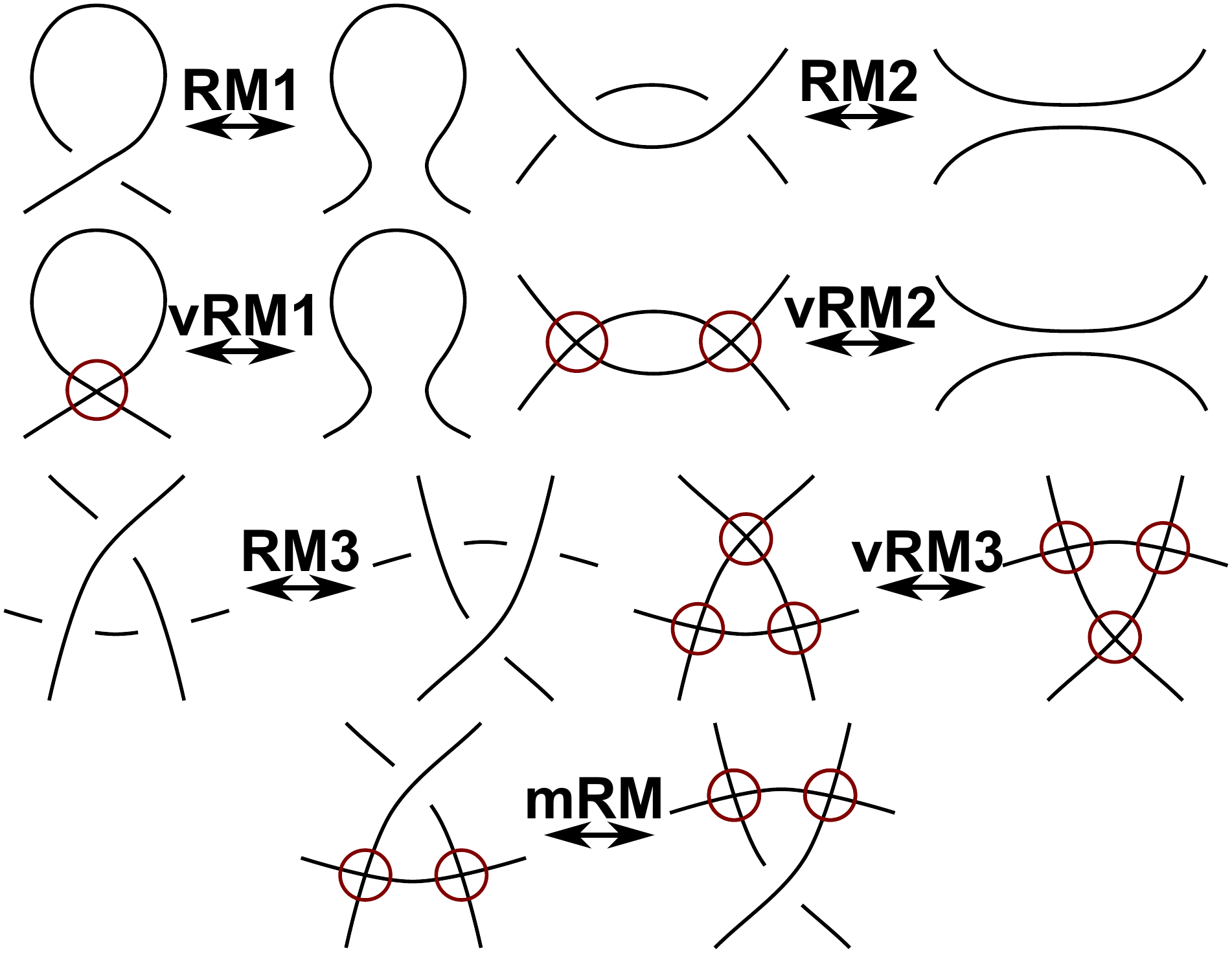}
  \caption{The generalised Reidemeister moves are the moves pictured plus their mirror image.}
  \label{figure0-1}
\end{figure}

In this paper we are also interested in \textit{virtual tangle diagrams} and \textit{virtual tangles}. The first ones are graphs embedded in a disk $D^2$ such that every vertex is either one valent or four valent. The four valent vertices are, like before, labeled with an \emph{overcrossing} $\slashoverback$, an \textit{undercrossing} $\backoverslash$ or a \textit{virtual crossing}. The one valent vertices are part of the boundary of $D^2$ and are called \emph{boundary points}. We call a virtual tangle diagram with $k$ one valent vertices a virtual tangle diagram with \textit{$k$-boundary points}.

A virtual tangle with $k$-boundary points is an equivalence class of virtual tangle diagrams with $k$-boundary points modulo the generalised Reidemeister moves and boundary preserving isotopies. We note that every of the moves of Figure \ref{figure0-1} can be seen as a virtual tangle diagram. Moreover, examples are pictured in Section \ref{v-tangles}, e.g. in Figure \ref{figure1-1}.

The notions of \textit{oriented} virtual tangle diagrams and an \textit{oriented} virtual tangles can be defined analogous, but latter modulo \textit{oriented} generalised Reidemeister moves and boundary preserving isotopies.

We use the short hand notation c- and v- for everything that starts with classical or virtual, e.g. c-knot means classical knot and v-crossing means virtual crossing.

If the reader is unfamiliar with the notion v-link or v-tangle, we refer to some introduction papers of L.~Kauffman and V.~Manturov, e.g. \cite{ka2} and \cite{kama} and the references therein.

A well-known invariant of v-links is the \textit{virtual Jones polynomial}, i.e. an extension of the Jones polynomial from c-links to v-links. Moreover, M.~Khovanov published in \cite{kh1} a \textit{categorification} of the Jones polynomial, i.e. a $q$-graded chain complex whose homotopy type is a c-link invariant and whose graded Euler characteristic is the Jones polynomial. This categorification is called \textit{Khovanov (co-)homology}.

The Khovanov complex is a highly studied invariant of c-links. D.~Bar-Natan gave a cobordism based exposition of this categorification in \cite{bn1}. D.~Bar-Natans picture has many nice properties, e.g. it can be done \textit{local}, i.e. it works for c-tangles as well.

It is a natural question if the Khovanov complex extents from c-links to v-links. The answer is positive, i.e. there is more then one \textit{categorification of the virtual Jones polynomial}, one over $\zet/2$ given by V.~Manturov in \cite{ma1}, one over $\zet$ given by V.~Manturov in \cite{ma2} and one over $\zet$ using cobordisms in the spirit of D.~Bar-Natan in \cite{tub}.

The case of v-links is more difficult (combinatorial) then the classical case. That is a reason why in \cite{tub} the D.~Bar-Natan approach was not extended to v-tangles. This is the main goal of the paper.

In this paper we will answer the question if the categorification extends from c-tangles to v-tangles positive, i.e. it is trivial to extend the notions over $\zet/2$ or in a trivial way (by setting open saddles to be zero). Latter has an obvious disadvantage, i.e. it is neither a ``good'' invariant of v-tangles nor can it be used to calculate bigger complexes by ``tensoring'' smaller pieces.  

But we give a local notion that is a strong invariant of v-tangles (see Section \ref{v-tangles}) and allows ``tensoring'' of smaller pieces (see Section \ref{circuit}).

As an application of this construction we prove in Section \ref{leedeg} that the degenerations of Lee's variant (see her paper \cite{lee}), i.e. that the homology in her case just ``counts'' the number of components, also extents to v-links. Note that latter result is quite surprising since the complex will contain lots of $0$-morphisms (our main observation is that it will also contain many isomorphisms).

We note that, given the semi-local properties of our construction, it is possible to write a ``fast computation'' algorithm in the sense of D.~Bar-Natan in \cite{bn3}
 
\section*{A brief summary}\label{summary}
\addcontentsline{toc}{section}{A brief summary}
Before we give a brief summary we recall the main ideas of the cobordism based construction of virtual Khovanov homology in \cite{tub}. We will assume that the reader is familiar with the notion of the classical Khovanov complex, e.g. the construction of the Khovanov cube based on so-called \emph{resolutions of the crossings}. There are many good introduction to classical Khovanov homology, e.g. a nice exposition of the classical Khovanov homology can be found in D.~Bar-Natan's paper \cite{bn2}. Moreover we assume that the reader is familiar with the cobordism based construction of the classical Khovanonv complex of D.~Bar-Natan in \cite{bn1}. Recall that latter involves the \textit{local relation} in Figure \ref{figure0-2}.

\begin{figure}[ht]
  \centering
     \includegraphics[scale=0.62]{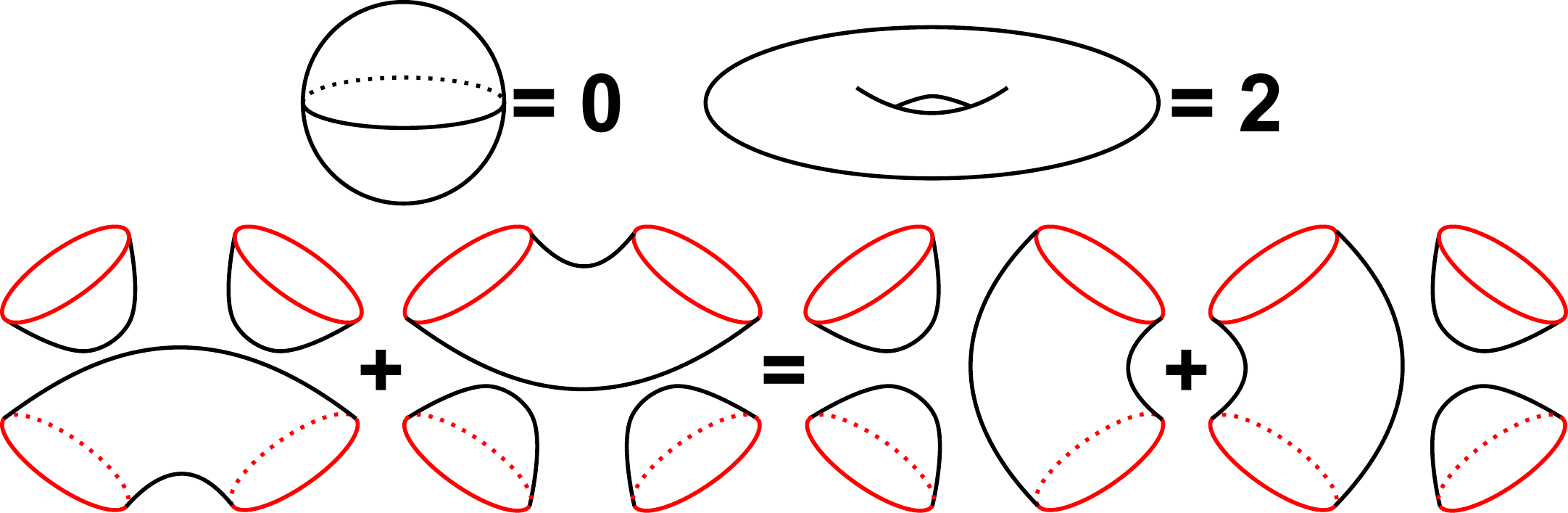}
  \caption{The local relations. A cobordism that contains a sphere $S$ should be zero, a cobordism that contains a torus $T$ should be two times the cobordism without the torus and the four tubes relation.}
  \label{figure0-2}
\end{figure}

To maintain readability, we repress in this short summary the more technical points about the construction, e.g. the placement of the signs. For more details see \cite{tub}. Note that this technical points are a main reason for the length of the paper. The following table shows the connection between the classical and the virtual case.
\begin{center}
\begin{tabular}{|c|c|c|}
\hline $\phantom{.}$ & \textbf{classical} & \textbf{virtual} \\ 
\hline \textbf{objects} & c-link resolutions & v-link resolutions \\ 
\hline \textbf{morphisms} & orientable cobordisms & possible non-orientable cobordisms \\ 
\hline \textbf{cobordisms} & embedded & immersed \\ 
\hline \textbf{decorations} & none & $+,-$ at the boundary \\ 
\hline \textbf{signs} & usual sign placement & rather technical sign placement \\ 
\hline 
\end{tabular} 
\end{center}
Hence, a main point in the construction of the virtual Khovanov complex is to say which saddles, i.e. morphisms, are orientable and which are non-orientable, how to place the decorations and how to place the signs. This is done in the following way.
\begin{itemize}
\item Every saddle either split one circle in a resolution (orientable, called \emph{comultiplication}, denoted $\Delta$. See Figure \ref{figure1-2} - upper row, fourth morphism), glues two circles in a resolution (orientable, called \emph{multiplication}, denoted $m$. See Figure \ref{figure1-2} - upper row, fifth morphism) or does not change the number of circles at all (non-orientable, called \emph{M\"obius cobordism}, denoted $\theta$. See Figure \ref{figure1-2} - upper row, last morphism).
\item Every saddle can be local denoted (up to a rotation) by a formal symbol $\smoothing\rightarrow\hsmoothing$ (both smoothings are neighbourhoods of the crossing). The \emph{glueing numbers}, i.e. the decorations, are now spread by \emph{choosing a formal orientation} for the resolution. We note that the construction will not depend on the choice.
\item After all resolutions have an orientation, a saddle could for example be of the form \du$\rightarrow$\ril. This is the \emph{standard form}, i.e. in this case every glueing number will be $+$.
\item Now spread the decorations as follows. Every boundary component gets a $+$ iff the orientation is like in the standard case and a $-$ otherwise. The \emph{degenerated} cases (everything non-alternating), e.g. \dd$\rightarrow$\rir, are the non-orientable surfaces and do not get any decorations.
\item There is a relation how the saddles behave if one changes the decorations. See for example Figure \ref{figure1-3} bottom row and Figure \ref{figure1-4} right part. Because of this relations we say that $m$ is of type $+1$, $\Delta$ of type $-1$ and $\theta$ of type $0$.
\item The signs are spread based on numbering of the circles in the resolutions and on a special x-marker for the crossings (rather technical, see \cite{tub} for details or Definition \ref{defin-saddlestructure} for a short introduction). A x-marker is for example pictured in Figure \ref{figure0-main}.
\end{itemize}
To construct the virtual Khovanov complex for v-tangles we need to extend the notions above in such a way that they still work for open cobordisms. A first generalisation is easy, i.e. we will still use immersed possible non-orientable surfaces with boundary decorations, but we allow vertical boundary components, e.g. the three v-Reidemeister cobordisms vRM1, vRM2 and vRM3 cobordisms in Figure \ref{figure0-3}.
\begin{figure}[ht]
  \centering
     \includegraphics[scale=1]{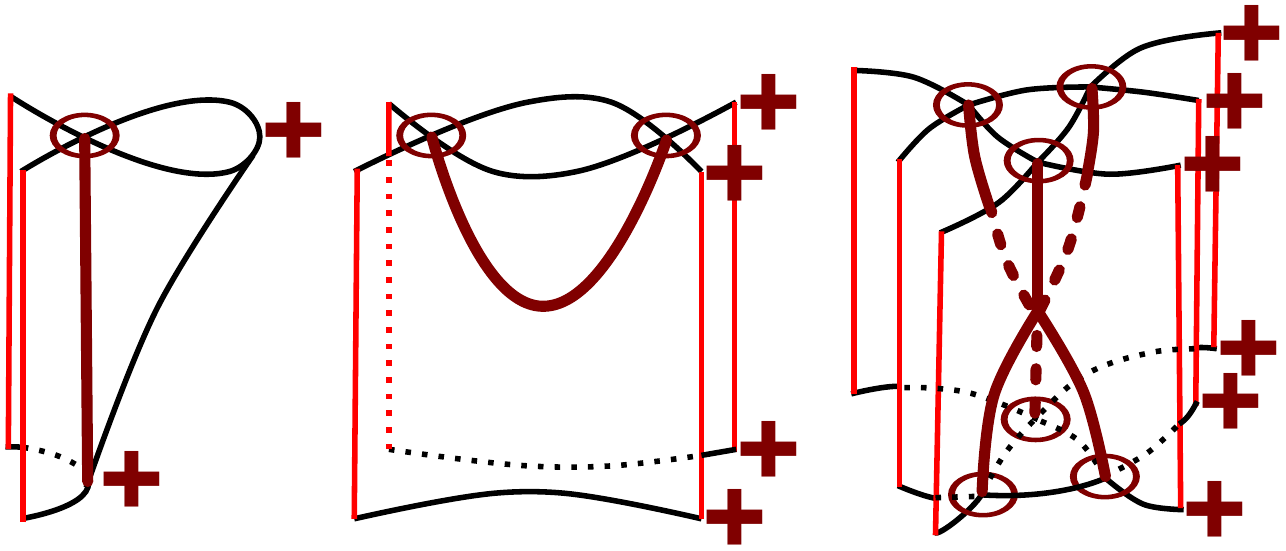}
  \caption{The virtual Reidemeister cobordisms. The red (vertical) lines are the vertical boundary components.}
  \label{figure0-3}
\end{figure}
One main point is the question what to do with the \emph{open saddles}, i.e. saddles with no closed boundary. See bottom row in Figure \ref{figure1-2}. A possible solution is to define them to zero. But this has two major problems. First the loss of information is big and second we would not have local properties like in the classical case (``tensoring'' of smaller parts) since an open saddle can, after closing some of his boundary circles, become either $m$, $\Delta$ or $\theta$. See Figure \ref{figure0-4}. 
\begin{figure}[ht]
  \centering
     \includegraphics[scale=0.5]{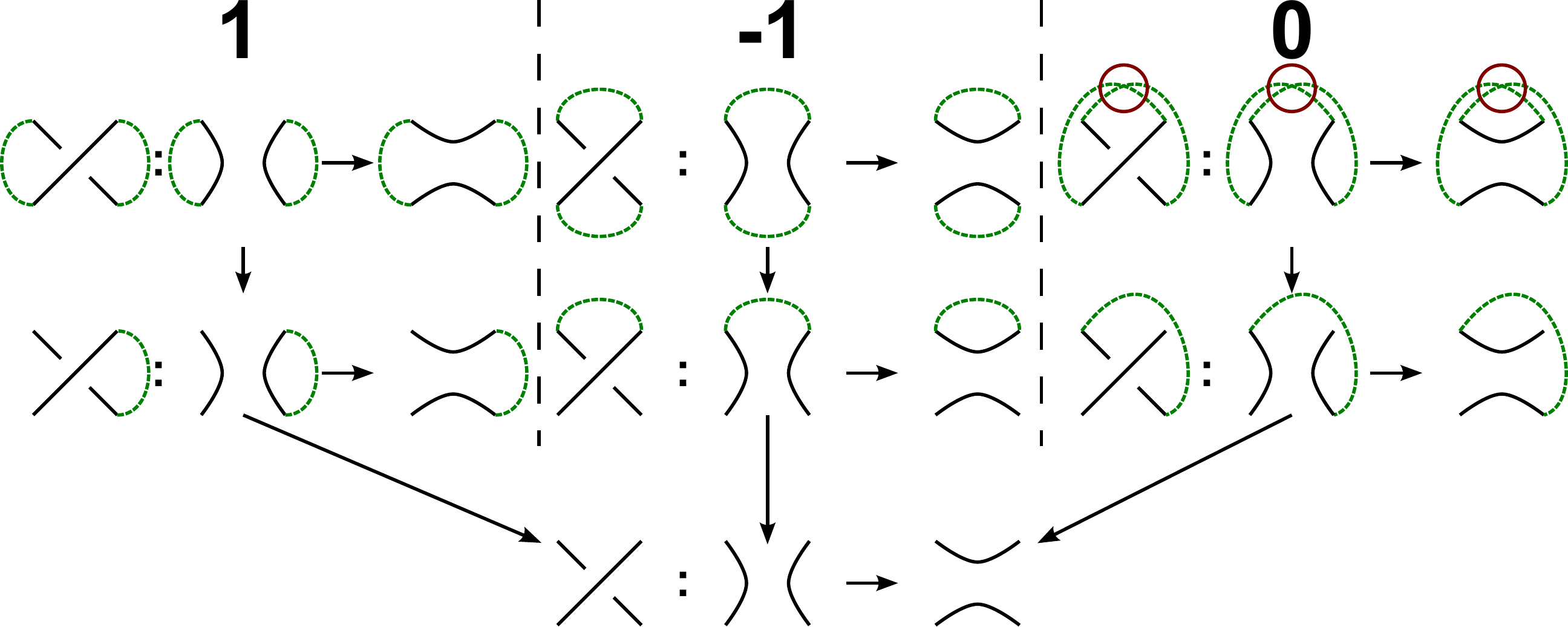}
  \caption{The saddle without closed boundary strings can come from any of the three closed cases.}
  \label{figure0-4}
\end{figure}
Hence, an information mod 3 is missing. We therefore consider morphisms with an \emph{indicator}, i.e. an element of the set $\{0,+1,-1\}$, e.g. see Figure \ref{figure1-2}. Then, after taking care of some technical difficulties, the concept extens from c-tangles to v-tangles.

The outline of the paper is as follows.
\begin{itemize}
\item In Section \ref{v-tangles} we define the notions of a \emph{v-tangle (diagram)} with $k$-boundary points. Moreover, we define the geometric category we are working in and define \emph{two different} versions of the virtual Khovanov complex. We show in Theorem \ref{theo-geoinvarianz} that both are invariants of v-tangles and we show in Theorem \ref{theo-marker} that both versions are equal for c-tangles and v-links.
\item In Section \ref{circuit} we discuss the local properties of the construction. We define  the notion of a \textit{decorated circuit diagram}, i.e. the virtual analogon of a planar diagram, and we show in Theorem \ref{theo-semiloc} that the construction has semi-local properties, i.e. the classical notion extents to a certain class of v-tangle diagrams. A main tool is the \emph{dot-calculus} which is sketched in Figure \ref{figure0-main}, i.e. change the decorations iff the orientations of the input diagram and the local piece mismatch (red dot), do not change them iff they match (green dot) and add in a bolt symbol (non-orientable) iff at the top and the bottom one gets a green and a red dot.
\item In Section \ref{leedeg} we show that the extension of Lee's variant of Khovanov homology still has the same ``degeneration'' as in the classical case, i.e. the homology just ``counts'' the number of v-link components. See Theorem \ref{theo-leedeg}. Note that this is rather surprising since the map $\theta$ will be a $0$-morphism if $2$ is invertible. Our main observation why the degeneration still holds is explained in Example \ref{beisp-leedeg}. In order to prove Theorem \ref{theo-leedeg} we use a special trick, i.e. we define two idempotents called \emph{down} and \emph{up}, go to the \emph{Karoubi envelope} and prove that there, using the semi-local construction of Section \ref{circuit}, the question simplifies to the case of counting \emph{non-alternating resolutions}. See Figure \ref{figure4-2}.
\end{itemize} 

\section*{Notation}\label{notation}
For a v-tangle diagram $T^k_D$ we call $\smoothing$ the \textit{$0$-} and $\hsmoothing$ the \textit{$1$-resolution} of the crossing $\slashoverback$. For an oriented v-tangle diagram $T^k_D$ we call $\overcrossing$ a \textit{positive} and $\undercrossing$ a \textit{negative} crossing. The \textit{number of positive crossings} is called $n_+$ and the \textit{number of negative crossings} is called $n_-$.

For a given v-tangle diagram $T^k_D$ with $n$-numbered crossings we define a collection of closed curves and open strings $\gamma_a$ in the following way: Let $a$ be a word of length $n$ in the alphabet $\left\{0,1\right\}$. Then $\gamma_a$ is the collection of closed curves and open strings which arise when one performs a $a_i$-resolution at the $i$-th crossing of $T^k_D$ for all $i=1,\dots,n$. We call such a collection $\gamma_a$ the \textit{$a$-th resolution} of $T^k_D$.

We can choose an orientation for the different components of $\gamma_a$. We call such a $\gamma_a$ an \textit{orientated resolution}.

If we ignore orientations, then there are $2^n$ different resolutions $\gamma_a$ of $T^k_D$. We say a resolution has length $m$ if it contains exactly $m$ 1-letters. That is $m=\sum_{i=1}^na_i$.

For two resolutions $\gamma_a$ and $\gamma_{a'}$ with $a_r=0$ and $a'_r=1$ for one fixed $r$ and $a_i=a'_i$ for $i\neq r$ we define a \textit{saddle between the resolutions $S$}. This means: choose a small (no other crossing, classical and virtual, should be involved) neighbourhood $N$ of the $r$-th crossing and define a cobordism between $\gamma_a$ and $\gamma_{a'}$ to be the identity outside of $N$ and a saddle inside of $N$.

Sometimes we need a so-called \emph{spanning tree argument}, i.e. choose a spanning tree of a cube (like in Figure \ref{figure0-5}) and change e.g. orientations such that the edges of the tree change in a suitable way, starting at leafs and remove them and repeat.
\begin{figure}[ht]
  \centering
     \includegraphics[width=0.9\linewidth]{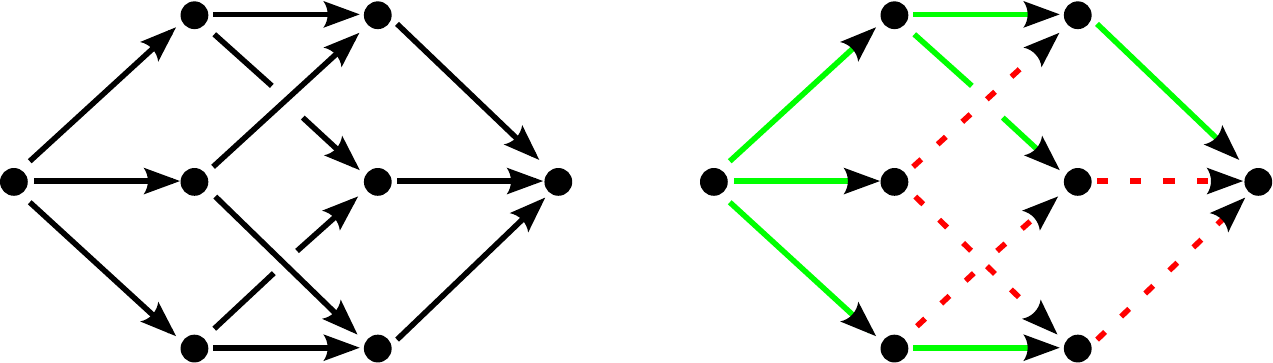}
  \caption{A Khovanov cube and a spanning tree of the cube (green edges).}
  \label{figure0-5}
\end{figure}

Note that any v-tangle diagram should be oriented in the usual sense. But with a slight abuse of notation, we will repress this orientation throughout the whole paper. Recall that these orientations are needed for the shifts in homology gradings, see for example \cite{bn2}.

\section*{Acknowledgement}
I wish to thank V.~Manturov for corrections and many helpful comments. Moreover i wish to thank A.~Kaestner and L.~Kauffman for suggestions and helpful comments helping me to write a computer program for calculations in the future.

\section{The geometric complex for virtual tangles}\label{v-tangles}
In this section we describe how the virtual Khovanov complex can be extended in such a way that it also works for v-tangles and not just for v-links. Our extension follows the ideas of D.~Bar-Natan in \cite{bn1}.

First we define the notion of a \textit{virtual tangle (diagram)}, called \textit{v-tangle (diagram)}.
\begin{defin}[Virtual tangle]\label{defin-vtangle}
A \textit{virtual tangle diagram with $k\in\nat$ boundary points} $T^k_D$ is a planar graph embedded in a disk $D^2$. This planar graph is a collection of \textit{usual vertices} and $k$-\textit{boundary vertices}. We also allow circles, i.e. closed edges without any vertices.

The usual vertices are all of valency four. Any of these vertices is either an overcrossing $\slashoverback$ or an undercrossing $\backoverslash$ or a virtual crossing. Latter is marked with a circle. The boundary vertices are of valency one and are part of the boundary of $D^2$.

We call the crossings $\slashoverback$ and $\backoverslash$ \textit{classical crossings} or just \textit{crossings} and a virtual tangle diagram without virtual crossings a \textit{classical tangle diagram}.

A \textit{virtual tangle with $k\in\nat$ boundary points} $T^k$ is an equivalence class of virtual tangle diagrams $T^k_D$ module boundary preserving isotopies and \textit{generalised Reidemeister moves} (see Figure \ref{figure0-1}).

We call a virtual tangle $T^k$ \textit{classical} if the set $T^k$ contains a classical tangle diagram. A v-circle is a circle without classical crossings and a v-string is a string starting and ending at the boundary without classical crossings. We call a v-circle/v-string without virtual crossings a \textit{c-circle/c-string}.

The \textit{closure of a v-tangle diagram with *-marker} $\mathrm{Cl}(T^k_D)$ is a v-link diagram which is constructed by capping of neighbouring boundary points (starting from a fixed point marked with the *-marker and going counter-clockwise) without creating new virtual crossings. For an example see Figure \ref{figure1-1}.

There are exactly two, maybe non equivalent, closures of any v-tangle diagram. In the figure below the two closures are pictured using green edges.

\begin{figure}[ht]
  \centering
     \includegraphics[scale=0.4]{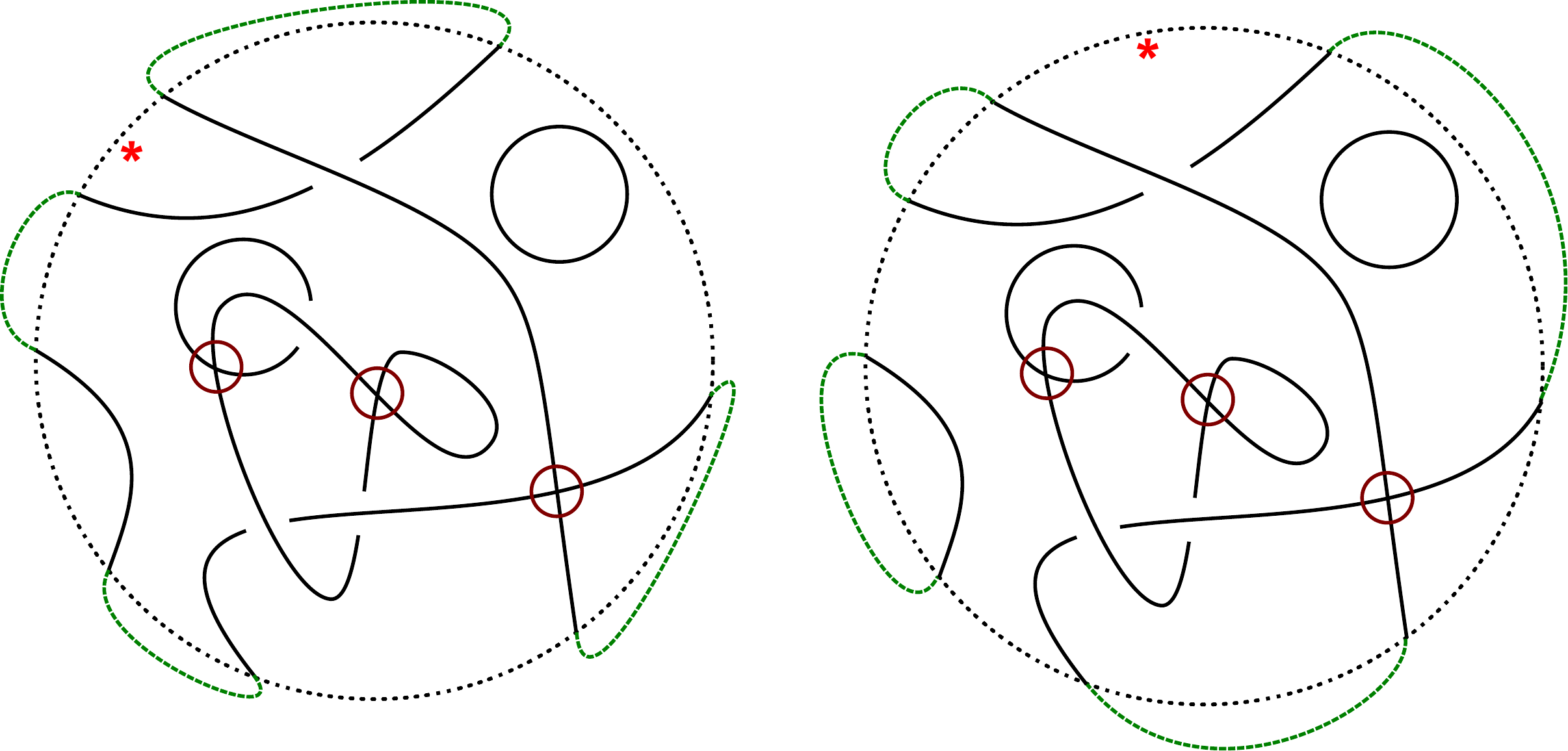}
  \caption{A v-tangle and the two different closures.}
  \label{figure1-1}
\end{figure}
\end{defin}

The notions of an \textit{oriented} virtual tangle diagram and of an \textit{oriented} virtual tangle are defined analogue. The latter modulo \textit{oriented} generalised Reidemeister moves and boundary preserving isotopies. From now on every v-tangle (diagram) should be oriented. But we repress this notion to avoid confusion with the (more important) notions of oriented resolutions and orientations for circuit diagrams (see Section \ref{circuit}).

Moreover, a \textit{decorated} v-tangle diagram is a v-tangle diagram with orientations for all of its circles/strings. 

We discuss the category of \textit{open cobordisms with boundary decorations}. It is almost the same as in \cite{tub}, but the corresponding cobordisms could be open, i.e. they could have vertical boundary components, and are decorated with an extra information, i.e. a number in the set $\{0,+1,-1\}$. We picture the number $0$ as a bolt.  

If the reader is unfamiliar with the notions and relations of cobordisms we suggest to check for example the book of J.~Kock \cite{ko}.

\begin{defin}[The category of open cobordisms with boundary decorations]\label{defin-category}
Let $k\in\nat$ and let $R$ be a commutative and unital ring.

Our category should be $R-$pre-additive. The symbol $\amalg$ should denote the disjoint union (the coproduct in our category).

\textbf{The objects:}

The \textit{objects} \textit{$\Ob(\ucob_R(k)$}) are numbered v-tangles diagrams with $k$ boundary points without classical crossings. We denote the objects as $\mathcal O=\coprod_{i\in I}\mathcal O_i$. Here $\mathcal O_i$ are the v-circles or v-strings and $I$ is a finite, ordered index set.

The objects of the category should be unique up to \textit{boundary preserving planar isotopies} of planar graphs.

\textbf{The generators:}

The \textit{generators} of $\Mor(\ucob_R(k))$ are the cobordisms from Figure \ref{figure1-2}. The cobordisms pictured are all between c-circles or c-strings. 

\begin{figure}[ht]
  \centering
     \includegraphics[scale=0.5]{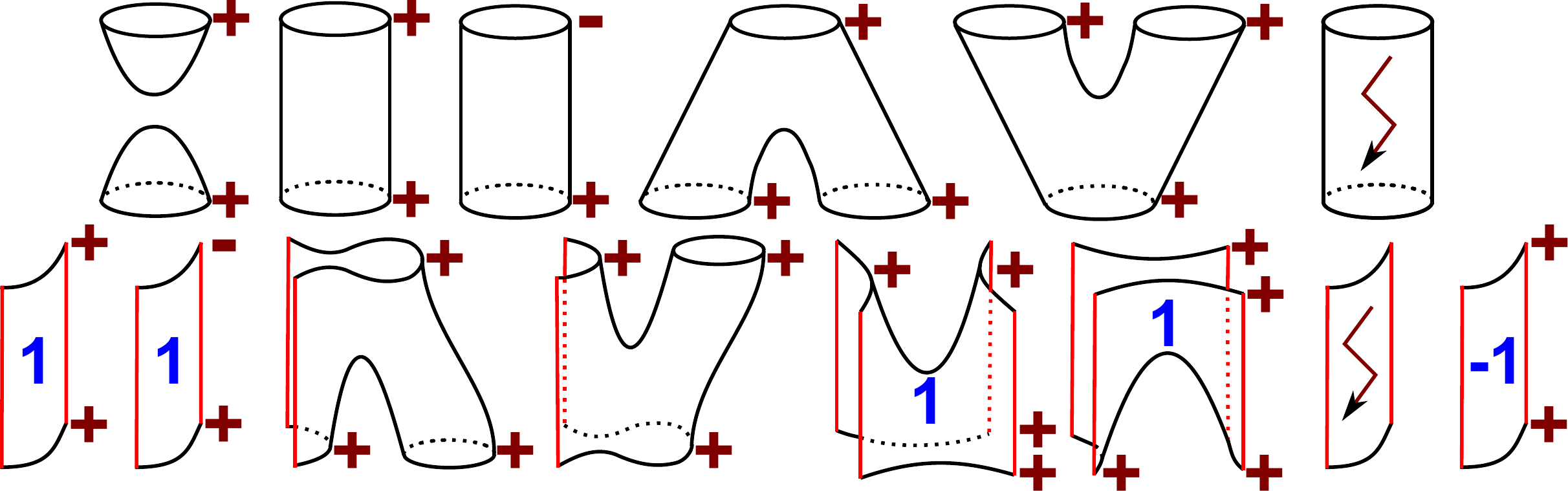}
  \caption{The generators for the set of morphisms. The cobordism on the upper right is the M\"obius cobordism, i.e. a two times punctured projective plane.}
  \label{figure1-2}
\end{figure}

Every generator has a decoration from the set $\{0,+1,-1\}$. We call this decoration the \textit{indicator} of the cobordisms. If no indicator is pictures it should be $+1$.

Every generator with a decoration $\{+1,-1\}$ has extra decorations from the set $\{+,-\}$ at every horizontal boundary component. We call these decorations the \textit{gluing numbers} of the cobordism.

We consider these cobordisms up to boundary preserving homeomorphisms (as abstract surfaces). Hence, between circles or strings with v-crossings the generators are the same up to boundary preserving homeomorphisms immersed into $D^2\times[-1,1]$.

The vertical boundary components are pictured in red.

We denote the different generators (from left to right; top row first) $\iota_+$ and $\varepsilon^+$, $\Id^+_+$ and $\Phi^-_+$, $\Delta^+_{++}$, $m^{++}_+$ and $\theta$, $\Id(1)^+_+$ and $\Phi(1)^-_+$, $S^{+}_{++}$ and $S^{++}_{+}$, $S(1)^{++}_{++}$, $\theta$ and $\Id(-1)^+_+$.

\textbf{The morphisms:}

The \textit{morphisms} \textit{$\Mor(\ucob_R(k))$} are cobordisms between the objects in the following way. First we identify the collection of numbered v-circles or v-strings with circles or strings immersed into $D^2$.

Given two objects $\mathcal O_1,\mathcal O_2$ with $k_1,k_2\in\nat$ numbered v-circles or v-strings, then a morphism $\mathcal C\colon\mathcal O_1\to\mathcal O_2$ is a surface immersed in $D^2\times[-1,1]$ whose non vertical boundary lies only in $D^2\times\{-1,1\}$ and is the disjoint union of the $k_1$ numbered v-circles or v-strings from $\mathcal O_1$ in $D^2\times\{1\}$ and the disjoint union of the $k_2$ numbered v-circles or v-strings from $\mathcal O_2$ in $D^2\times\{-1\}$.

The morphisms should be generated (as abstract surfaces) by the generators from above (see Figure \ref{figure1-2}).

\textbf{The decorations:}

Every morphisms has an \textit{indicator} from the set $\{0,+1,-1\}$.

Moreover, every morphism $\mathcal C\colon\mathcal O_1\to\mathcal O_2$ in $\Mor(\ucob_R(k))$ is a cobordism between the numbered v-circles or v-strings of $\mathcal O_1$ and $\mathcal O_2$. Let us say that the v-circles or v-strings of $\mathcal O_1$ are numbered $i\in\{1,\dots,l_1\}$ and the v-circles or v-strings from $\mathcal O_2$ are numbered for $i\in\{l_1,\dots,l_2\}$.

Every cobordism with $+1,-1$ as an indicator should have a decoration on the $i$-th boundary circle. This decoration is an element of the set $\{+,-\}$. We call this decoration of the $i$-th boundary component the \textit{$i$-th gluing number} of the cobordism.

Hence, the morphisms of the category are pairs $(\mathcal C,w)$. Here $\mathcal C\colon\mathcal O_1\to\mathcal O_2$ is a cobordism from $\mathcal O_1$ to $\mathcal O_2$ immersed in $D^2\times[-1,1]$ and $w$ is a string of length $l_2$ in such a way that the $i$-th letter of $w$ is the $i$-th gluing number of the cobordism and the last latter is the indicator or $w=0$ if the cobordism has $0$ as indicator.

\textbf{Short hand notation:}

We denote a morphism $\mathcal C$ which is a connected surfaces as $\mathcal C^{u}_{l}(\mathrm{in})$. Here $u,l$ are words in the alphabet $\{+,-\}$ in such a way that the $i-$th character of $u$ (of $l$) is the gluing number of the $i-$th circle of the upper (of the lower) boundary. The number $\mathrm{in}$ should be the indicator.

The construction above ensures that this notation is always possible. Therefore we denote an arbitrary morphism $(\mathcal C,w)$ as
\[
\mathcal C(\pm 1)=\mathcal C^{u_1}_{l_1}\amalg\cdots\amalg\mathcal C^{u_{k}}_{l_{k}}(\pm 1).
\]
Here $\mathcal C^{u_i}_{l_i}$ are its connected components and $u_i,l_i$ are words in $\{+,-\}$.

For a morphism with $0$ as indicator we do not need any boundary decorations. With a slight abuse of notation, we denote all these cobordisms like the non orientable cobordisms as $\theta$.

\textbf{The relations:}
\begin{figure}[p]
\begin{minipage}[c]{6,9cm}
	\centering
	\includegraphics[scale=0.25]{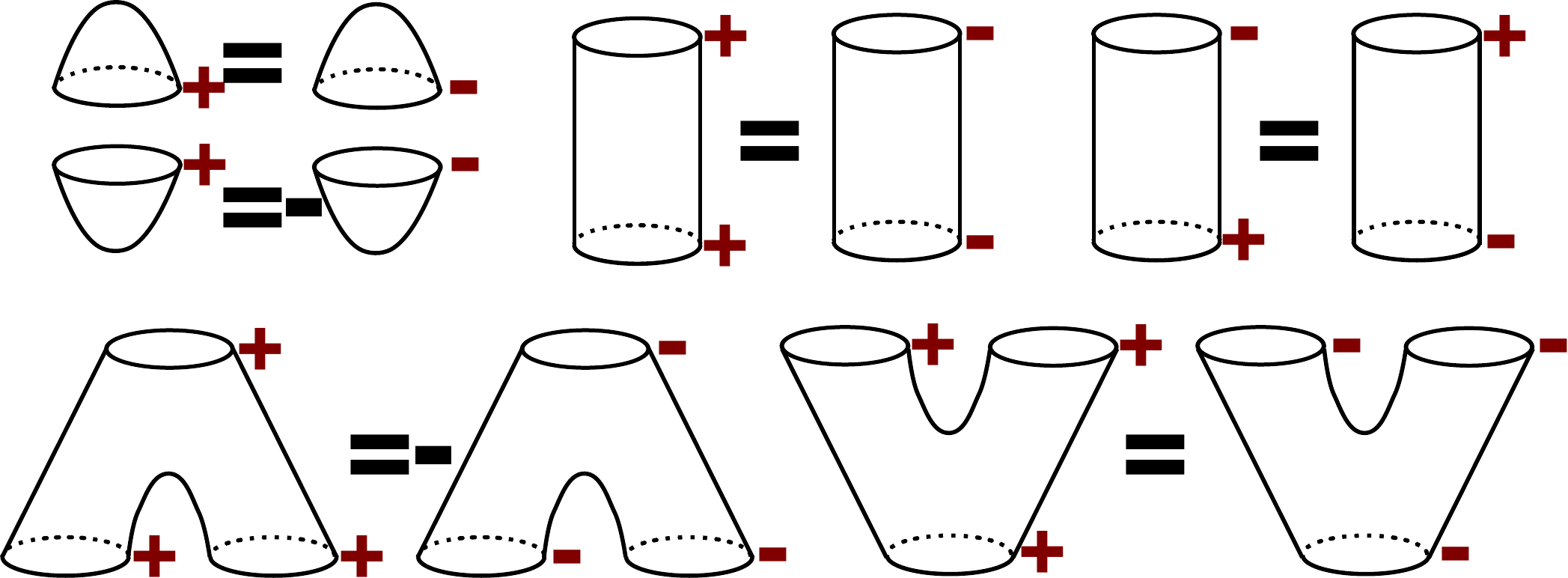}
	\caption{(1) The first combinatorial relations.}
	\label{figure1-3}
\end{minipage}
\begin{minipage}[c]{6,9cm}
	\centering
	\includegraphics[scale=0.25]{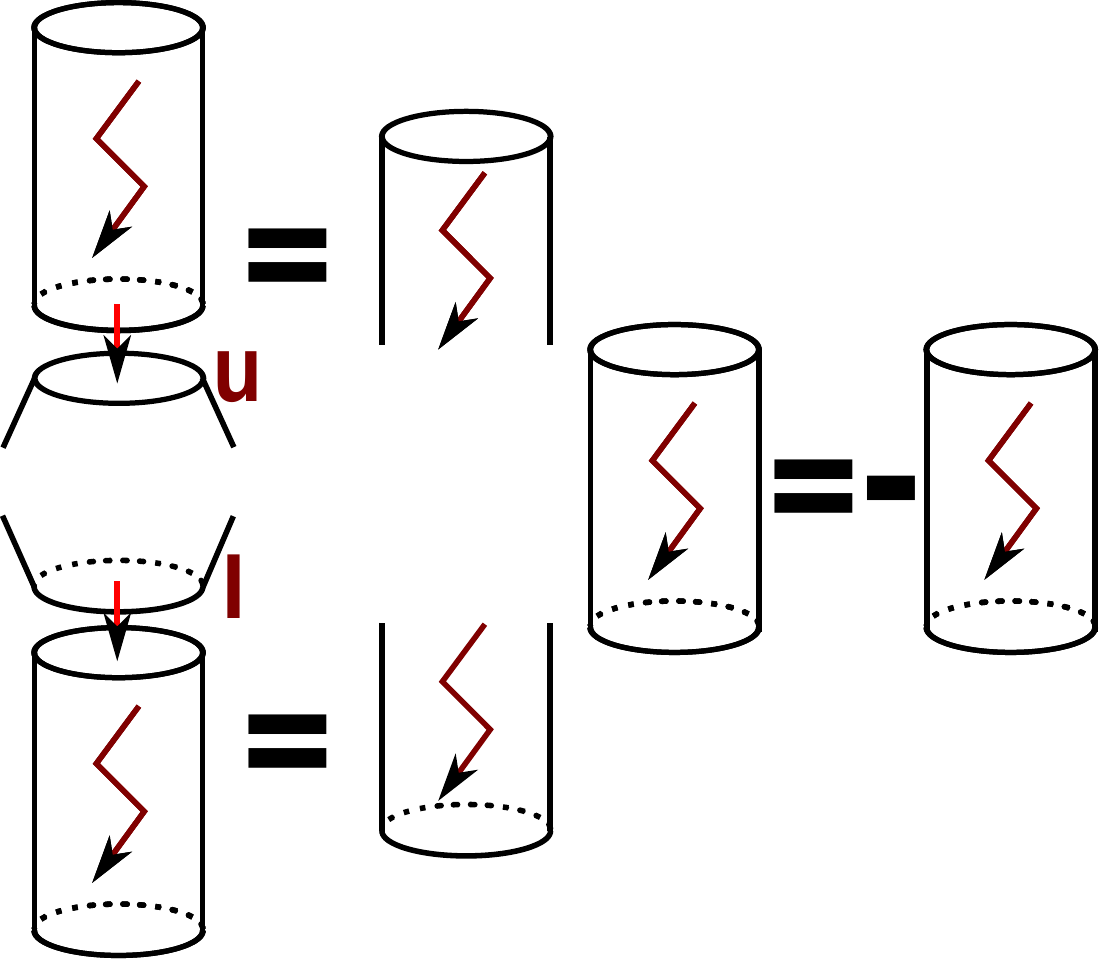}
	\caption{(2) The second combinatorial relations.}
	\label{figure1-4}
\end{minipage}
\begin{minipage}[c]{6,9cm}
	\centering
	\includegraphics[scale=0.25]{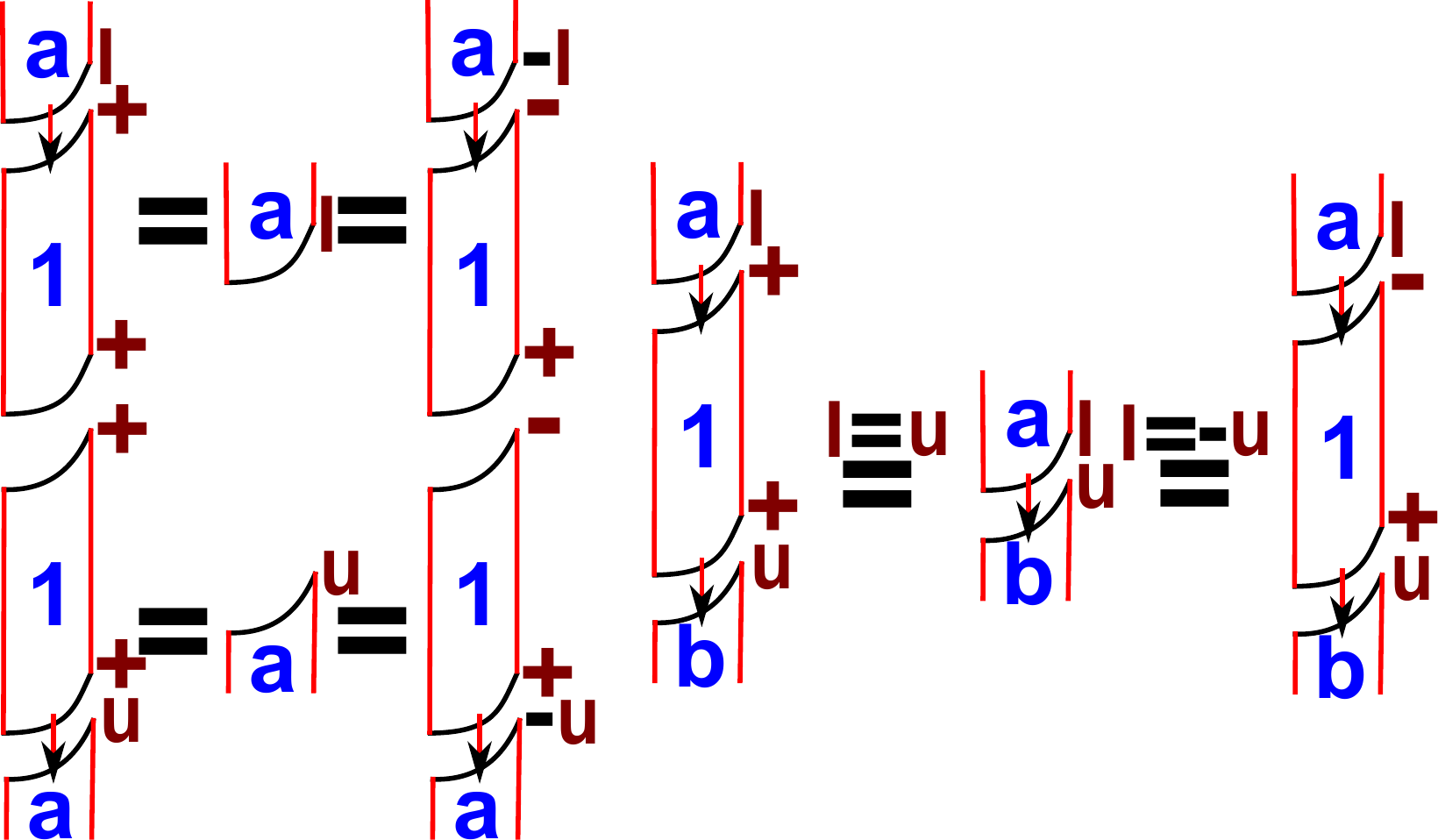}
	\caption{(3) The third combinatorial relations.}
	\label{figure1-5}
\end{minipage}
\begin{minipage}[c]{6,9cm}
	\centering
	\includegraphics[scale=0.25]{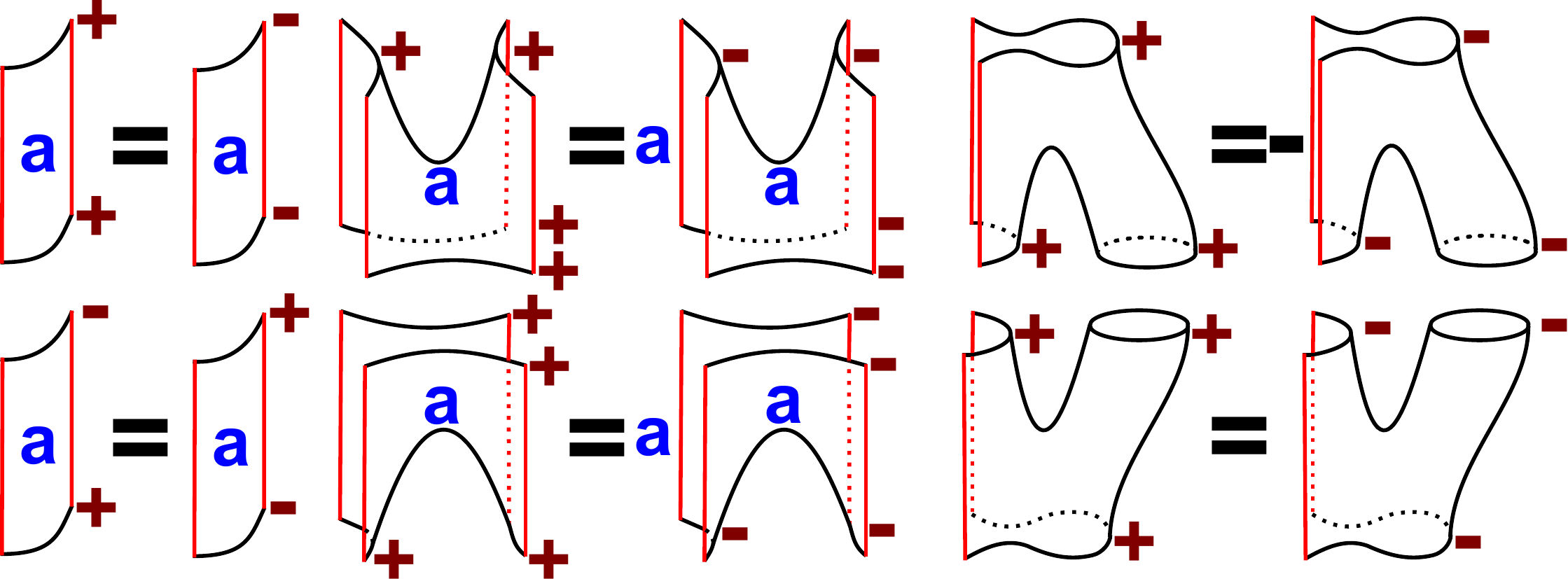}
	\caption{(4) The fourth combinatorial relations.}
	\label{figure1-6}
\end{minipage}
\begin{minipage}[c]{6,9cm}
	\centering
	\includegraphics[scale=0.25]{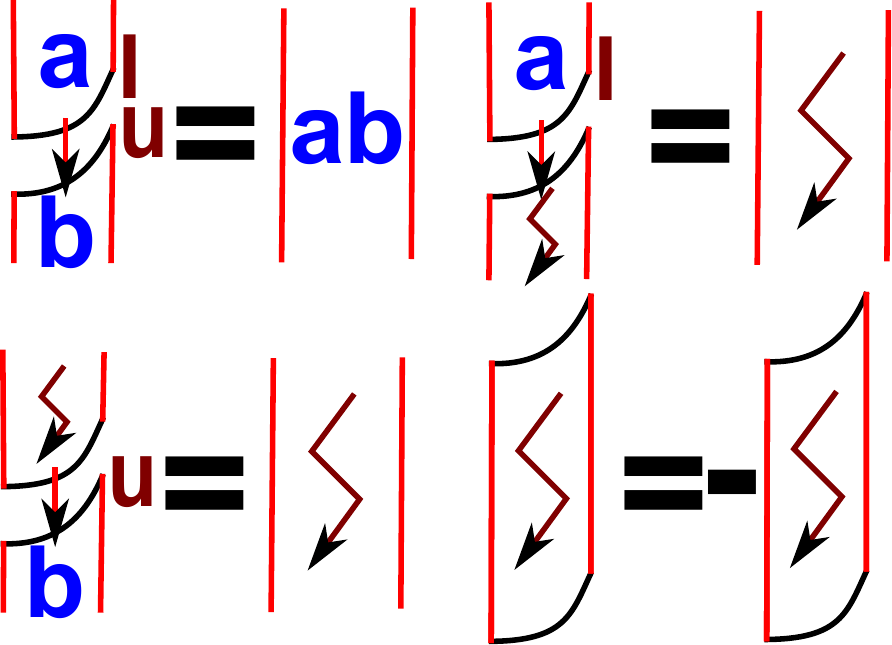}
	\caption{(5) The fifth combinatorial relations.}
	\label{figure1-7}
\end{minipage}
\begin{minipage}[c]{6,9cm}
	\centering
	\includegraphics[scale=0.25]{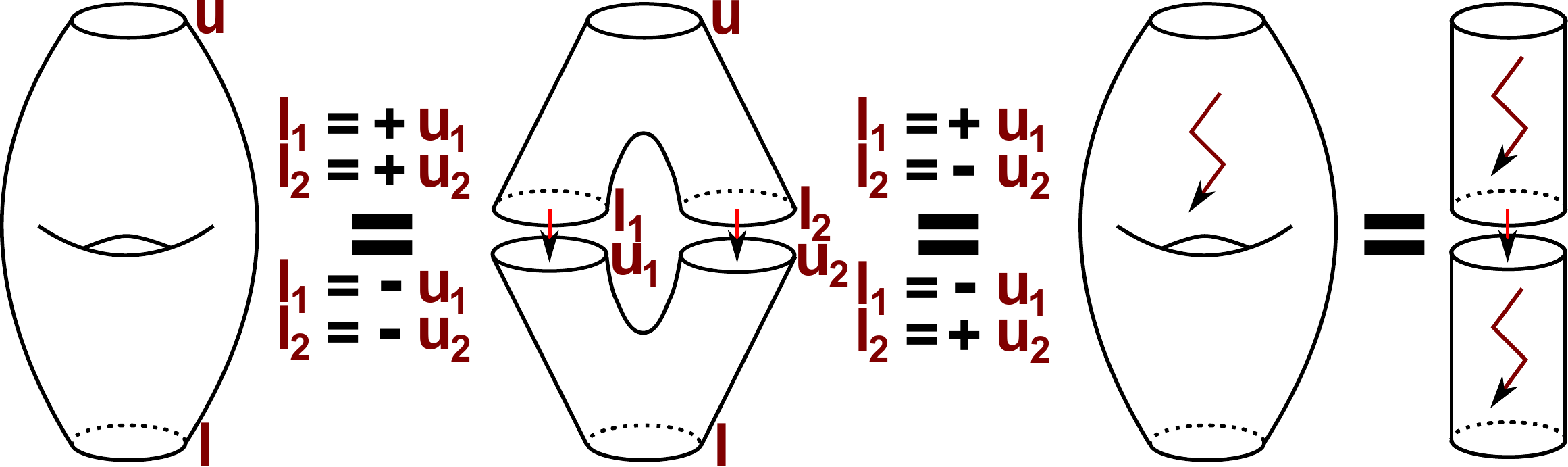}
	\caption{(6) The torus and M\"obius relations.}
	\label{figure1-8}
\end{minipage}
\begin{minipage}[c]{13,9cm}
	\centering
	\includegraphics[scale=0.25]{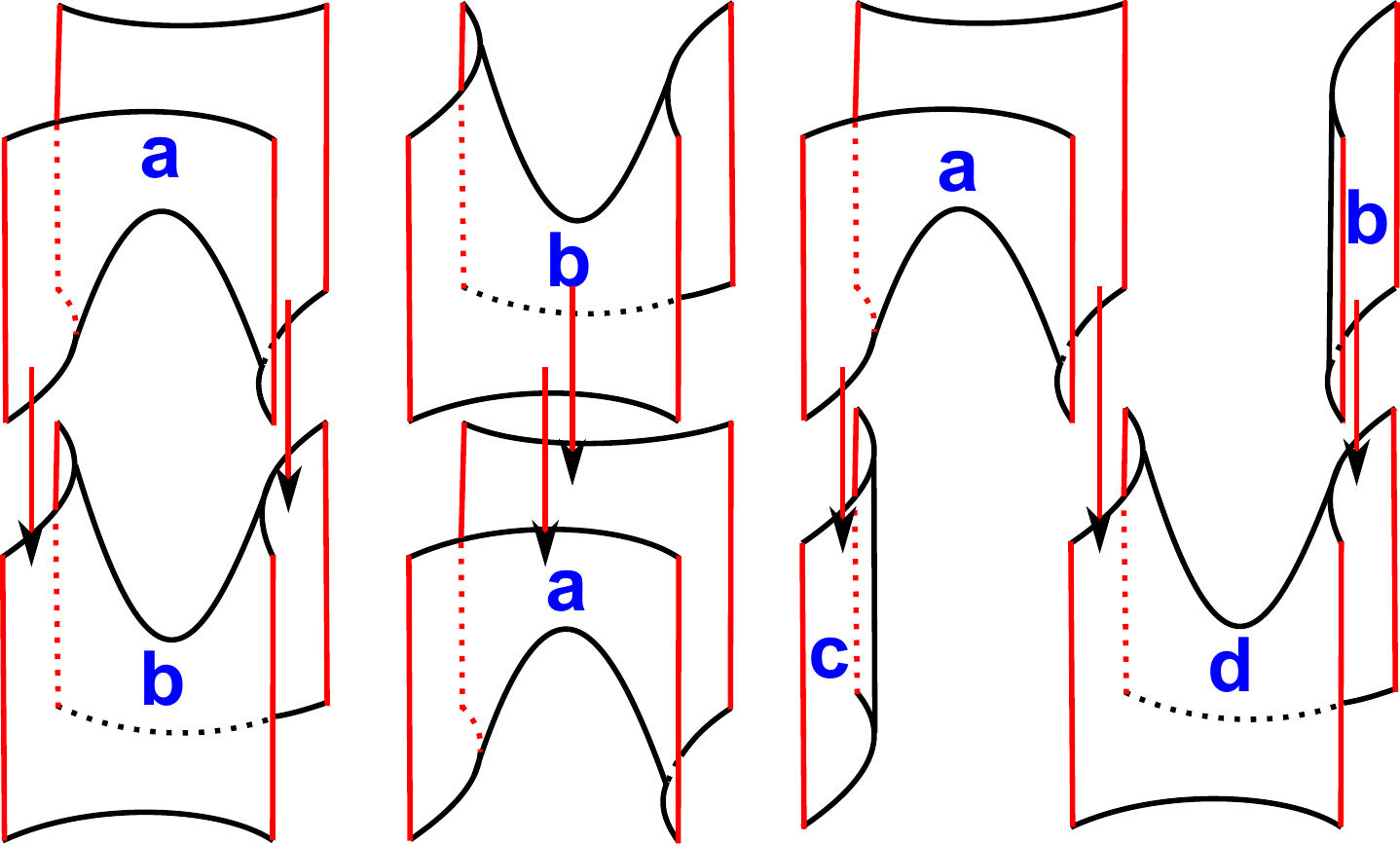}
	\caption{(7) The gluing in these three cases is described by the gluing numbers, i.e. if there is an odd number of different gluing numbers then the indicator should be $0$ and just the product otherwise.}
	\label{figure1-9}
\end{minipage}
\end{figure}
There are different relations for the cobordisms, namely \textit{topological relations} and \textit{combinatorial relations}. The latter relations are describe by the gluing numbers and indicators of the cobordisms and the gluing of the cobordisms. The topological relations are not pictured but it should be clear how they should work. See for example J.~Kock in \cite{ko} and our paper in \cite{tub}.

The monoidal structure should be induced by the disjoint union $\amalg$. The category is a semi-strict monoidal category. Moreover the category is symmetric. Hence, by Mac Lane's coherence theorem (see \cite{mac}) we can assume that the category is a strict, symmetric, monoidal category.

We define the full category $\ucob_R(\omega)$ to be the category whose objects are $\bigcup_{k\in\nat}\Ob(\ucob_R(k))$ and whose morphisms are $\bigcup_{k\in\nat}\Mor(\ucob_R(k))$.
\end{defin}
Note the following collection of formulas that follow from the relations. Recall that $\Phi^-_+$ and $\Phi(1)^-_+$ change the decorations and that $\theta$ and $\Id(-1)^+_+$ change the indicators. With a slight abuse of notation we repress to write tensors if it is not neccessary, i.e. for the indicator changes. Moreover, since $\Phi^-_+$ and $\Phi(1)^-_+$ satiesfy similar formulas, we only write down the equations for $\Phi^-_+$.
\begin{lemm}\label{lemm-calc}
Let $\mathcal O,\mathcal O'$ be two objects in $\ucob_R(k)$. Let $\mathcal C\colon\mathcal O\to\mathcal O'$ be a morphism that is connected, has $\mathrm{in}\in\{0,+1,-1\}$ as an indicator and $u$ and $l$ as decorated boundary strings. Then we have the following identities. We write $\mathcal C=\mathcal C^u_l(\mathrm{in})$ as a short hand notation if the indicators and gluing numbers do not matter.
\begin{itemize}
\item[(a)] $\mathcal C\circ\Id(-1)^+_+=\Id(-1)^+_+\circ\mathcal C$ (indicator changes commute).
\item[(b)] $\mathcal C\circ\theta=\theta\circ\mathcal C$ ($\theta$ commutes).
\item[(c)] $\mathcal C(0)\circ\Phi^-_+=\Phi^-_+\circ\mathcal C(0)$ (first decoration commutation relation).
\item[(d)] Let $u',l'$ denote the decoration change at the corresponding positions of the words $u,l$. Then we have \begin{align*}
\mathcal C(\pm 1)^u_l\circ\Id^+_+\otimes\dots\otimes\Phi^-_+\otimes\dots\otimes\Id^+_+ &=\mathcal C(\pm 1)^{u'}_l=\pm\mathcal C(\pm 1)^u_{l'}\\ &=\pm\Phi^-_+\otimes\dots\otimes\Id^+_+\otimes\dots\otimes\Phi^-_+\circ\mathcal C(\pm 1)^u_l
\end{align*}(second decoration commutation relation).
\end{itemize}
\end{lemm}
\begin{proof}
Everything follows by a straightforeward usage of the relations in Definition \ref{defin-category} above. 
\end{proof}
Next we define some short hand notations for the different categories we use. For any pre-additive category $\mathfrak C$ we define $\mat(\mathfrak C)$ as the \textit{category of formal matrices} and $\kom(\mathfrak C)$ as the \textit{category of formal chain complexes} over $\mathfrak C$. If the reader is unfamiliar with the notations we refer to \cite{bn1}.

\begin{defin}\label{defin-category2}
We call $\ukob_R(k)$ the category $\kom(\mat(\ucob_R(k)))$. Here our objects are formal chain complexes of formal direct sums in the category of (possible non orientable) open cobordisms with boundary decorations. Then we define $\ukob_R(k)_h$ as the category $\ukob_R(k)$ modulo formal chain homotopy.

Moreover, we define $\ucob_R(k)_l$, which has the same objects as the category $\ucob_R(k)$, but the morphisms are all modulo the local relations from Figure \ref{figure0-2}.

Furthermore we define $\ukob_R(k)_l$ and $\ukob_R(k)_{hl}$ in the obvious sense. The notations $\ucob_R(k)_{(h)(l)}$ or $\ukob_R(k)_{(h)(l)}$ mean that we consider all possible cases, namely with or without a $h$ and with or without a $l$. Moreover, we use the same notions if $k=\omega$.
\end{defin}

We are now able to define the \textit{geometric complex of a v-tangle diagram} $T^k_D$. For this construction we use our notations for the \textit{saddle decorations} and \textit{saddle signs} of v-link diagrams $L_D$ from section two in \cite{tub}. The reader unfamiliar with the notations may check the corresponding definitions. We only repeat these two important constructions in a very brief summary.

The only new notion in the definition below is the spreading of the \textit{indicator of the saddle}.

\begin{defin}[The saddle decorations]\label{defin-saddlestructure}
Let $L_D$ be a v-link diagram. Moreover, let $S\colon\gamma_a\rightarrow\gamma_b$ be a saddle between the numbered, orientated resolutions $\gamma_a,\gamma_b$.

Every saddle can be viewed as an formal symbol $\smoothing\rightarrow\hsmoothing$ together with a formal local orientation of these strings which is equal to the orientations of $\gamma_a,\gamma_b$. We say the saddle of the form \du$\rightarrow$\ril\, is of \textit{standard form}.

Moreover, choose a position for the \emph{x-marker} at any crossing and a numbering of the v-circles in the resolutions. Recall that the x-marker is just a choice in which direction one rotates a saddle $S\colon\smoothing\to\hsmoothing$ until one can compare it to the saddle in standard form. The complex will be well-defined for any choice and different choices give chain isomorphic complexes (section two in \cite{tub}).

For every orientable saddle the \textit{saddle sign} is the product of two numbers, both either $1$ or $-1$. If the x-marker is on the lower numbered saddle, then the first number should be $1$ and $-1$ otherwise. Moreover, if the number of v-circles with a higher number then the x-marked v-circle is odd, then the second number should be $1$ and $-1$ otherwise. The reader should check that this definition corresponds to the one in section two of \cite{tub}.  

Now we spread the \textit{saddle decorations} as follows:
\begin{itemize}
\item[(i)] If the saddle is non orientable, then we do not need any decorations.
\item[(ii)] If the saddle is orientable, then we compose the standard saddles $m^{++}_+$ or $\Delta^+_{++}$ with $\Phi^-_+$ for every not standard oriented string.
\end{itemize}
The saddle should have a formal minus sign iff the saddle sign is negative. Every orientable surface should carry an indicator $+1$ iff the number of upper boundary components of the saddle is two and a $-1$ iff the number is one. Every non orientable saddle gets a zero as indicator.

Everything together, i.e. boundary decorations, saddle sign and indicator, is denoted as the \textit{saddle decoration} of $S$. Beware that many choices are involved. But the do not change the complex up to chain isotopies (Lemma \ref{lemma-everythingfine}).

Now let $T^k_D$ be a tangle diagram with a *-marker on the boundary and let $\mathrm{Cl}(T^k_D)$ be the closure of the diagram. The \textit{saddle decoration} of the saddles of $T^k_D$ should be the ones induced by the saddle decorations of the closure.
\end{defin}

\begin{defin}[Geometric complex]\label{defin-geocomplex}
For a v-tangle diagram $T^k_D$ with a *-marker on the boundary and with $n$ ordered crossings we define \textit{the geometric complex} $\bn{T^k_D}$ as follows:
\begin{itemize}
\item For $i\in\{0,\dots,n\}$ the $i-n_-$ \textit{chain module} is the formal direct sum of all oriented resolutions $\gamma_a$ of length $i$.
\item There are only morphisms between the chain modules of length $i$ and $i+1$.
\item If two words $a,a'$ differ only in exactly one letter and $a_r=0$ and $a'_r=1$, then there is a morphism between $\gamma_a$ and $\gamma_{a'}$. Otherwise all morphisms between components of length $i$ and $i+1$ are zero.
\item This morphism is a \textit{saddle} between $\gamma_a$ and $\gamma_{a'}$.
\item The saddles should carry the \textit{saddle decorations} from Definition \ref{defin-saddlestructure}.
\end{itemize}
Note that it is not clear at this point why we can choose the numbering of the crossings, the numbering of the v-circles and the orientation of the resolutions of the closure. Furthermore it is not clear why this complex is a well-defined chain complex.

We show in Lemma \ref{lemma-everythingfine} that the complex is independent of these choices, i.e. if $\bn{L_D}_1$ and $\bn{L_D}_2$ are well-defined chain complexes with different choices, then they are equal up to chain isomorphisms. The same lemma ensures that the complex is a well-defined chain complex.

Beware that the position of the *-marker is important for v-tangle diagrams. But Theorem \ref{theo-marker} ensures that the position is not important for classical v-tangles and v-links.

If it does not matter which of the possible two different chain complexes is which, i.e. it is just important that they could be different, then we denote them as $\bn{T^k_D}^*$ and $\bn{T^k_D}_*$ for a v-tangle diagram $T^k_D$ without a chosen marker position. 
 
For an example see Figure \ref{figure1-10}. This figure shows the virtual Khovanov complex of a v-tangle diagram with two different *-marker positions. The vertical arrow between them indicates that they are chain isomorphic.
\begin{figure}[ht]
  \centering
     \includegraphics[scale=0.4]{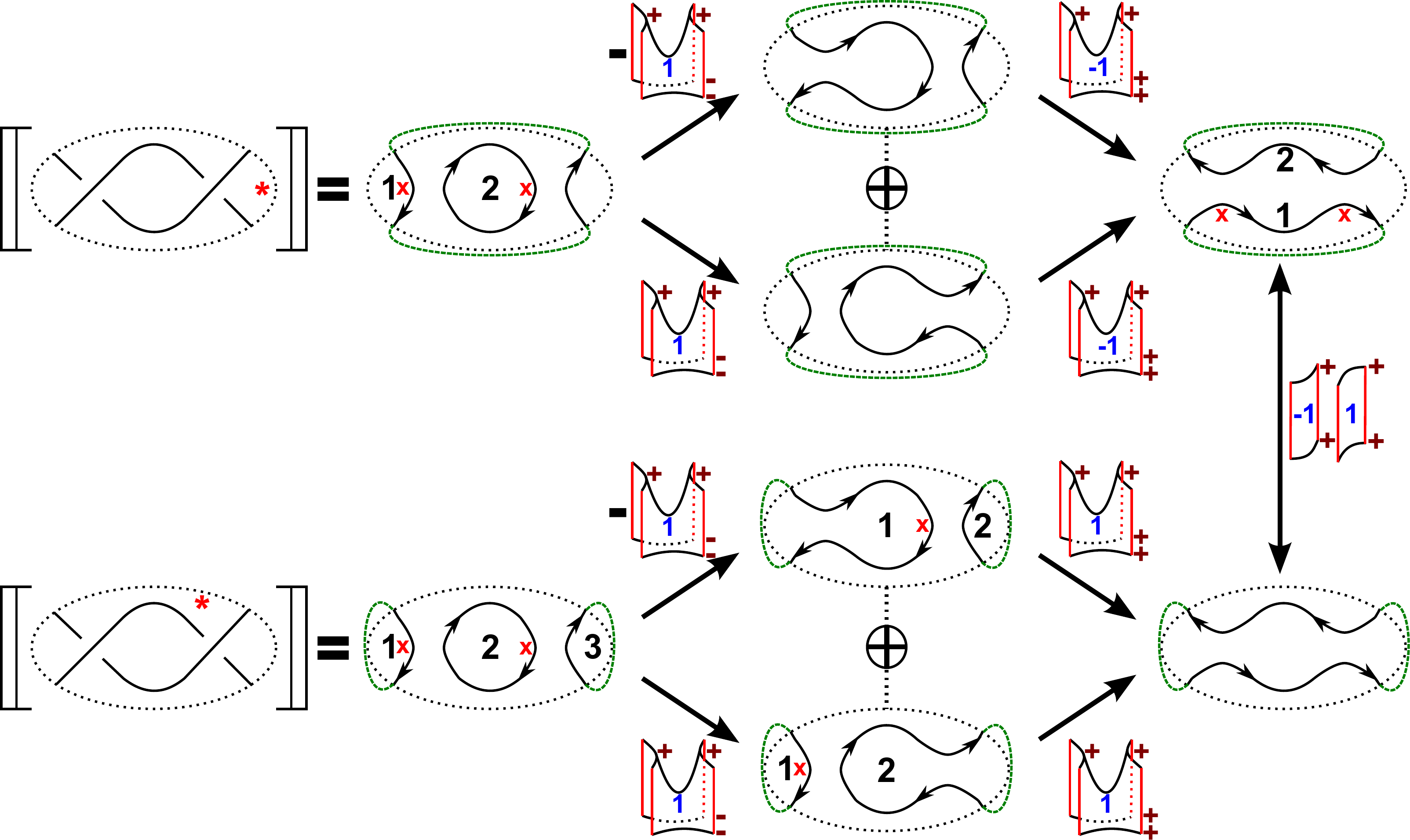}
  \caption{The complex of the same v-tangles with different *-marker positions. The two complexes are isomorphic.}
  \label{figure1-10}
\end{figure}
\end{defin}

\begin{lemm}\label{lemma-everythingfine}
Let $T^k_D$ be a v-tangle diagram with a *-marker and let $\bn{T^k_D}_1$ be its geometric complex from Definition \ref{defin-geocomplex} with arbitrary orientations for the resolutions of the closure. Let $\bn{T^k_D}_2$ be the complex with the same orientations for the resolutions except for one circle $c$ in one resolution $\gamma_a$. If a face $F_1$ from $\bn{T^k_D}_1$ is anti-commutative, then the corresponding face $F_2$ from $\bn{T^k_D}_2$ is also anti-commutative.

Moreover, if $\bn{T^k_D}_1$ is a well-defined chain complex, then it is isomorphic to $\bn{T^k_D}_2$, which is also a well-defined chain complex.

The same statement is true if the difference between the two complexes is the numbering of the crossings, the choice of the rotation for the calculation of the saddle signs or the fixed numbering of the v-circles of the closure. Moreover, the same is true for any rotations/isotopies of the v-tangle diagram.

Hence, we can speak of \textit{the} geometric complex $\bn{T^k_D}$ of the v-tangle diagram with a *-marker. The complex is a well-defined chain complex.
\end{lemm}

\begin{proof}
For v-tangle diagrams $T^k_D$ with $k=0$ the statement is the same as the corresponding statement in section two in \cite{tub}. Recall that the trick is to reduce all faces through a finite sequence of vRM1, vRM2, vRM3 and mRM moves (see Figure \ref{figure0-1}) and virtualisations (see Figure \ref{figure1-11}) to a finite number of different possible faces. Then one does a case-by-case check.

\begin{figure}[ht]
  \centering
     \includegraphics[scale=0.6]{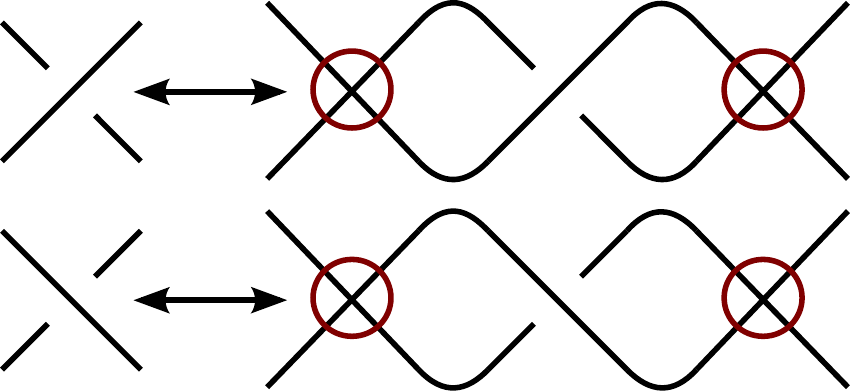}
  \caption{The two moves pictured are called virtualisations.}
  \label{figure1-11}
\end{figure}

Because the saddles in the two chain complexes are topological the same, we only have to worry about the decorations. But the decorations are spread based on the closure of the v-tangle diagram and the relations from Definition \ref{defin-category} are build in such a way that the open cases behave like the closed ones. Hence, we can use the statement for $k=0$ to finish the proof.
\end{proof}

The next theorem is very important but the proof itself is almost equal to the proof of D.~Bar-Natan in \cite{bn1} and our version in \cite{tub}. Therefore we skip the details.

\begin{theo}[The geometric complex is an invariant]
\label{theo-geoinvarianz}
Let $T^k_D,T'^k_D$ be two v-tangle diagrams with the same *-marker position which differs only through a finite sequence of isotopies and generalised Reidemeister moves. Then the complexes $\bn{T^k_D}$ and $\bn{T'^k_D}$ are equal in $\ukob(k)_{hl}$.
\end{theo}

\begin{proof}
Because the proof of D.~Bar-Natan in \cite{bn1} and therefore the corresponding proof in \cite{tub} is local we could copy the arguments from there. The Lemma \ref{lemma-everythingfine} guarantees that we can choose the numbering and orientations without changing anything up to chain isomorphisms.

Furthermore we use the statements of the Theorem \ref{theo-semiloc} and the Corollary \ref{koro-reidemeister} to finish the proof.

Beware that the chain homotopies in \cite{tub} should all carry $+1$ as indicator.
\end{proof}

We need some notions now. Let $T^k_D$ denote a v-tangle diagram. We call a part of $T^k_D$ a \emph{connected part} if it is connected as the four-valent graph on a surface with boundary which is created by replacing every c-crossing with a four-valent vertex and eliminating v-crossings by adding handles to a sphere with boundary. The genus of the surfaces should be high enough such that the graph has no v-crossings any more.

We call a connected part of a v-tangle diagram \emph{fully internal} if it is not adjacent to the boundary. See Figure \ref{figure1-12}. The left v-tangle diagram has one connected part, which is not fully internal, and the right v-tangle diagram has two connected parts, one fully internal and one not fully internal.

\begin{figure}[ht]
  \centering
     \includegraphics[scale=0.6]{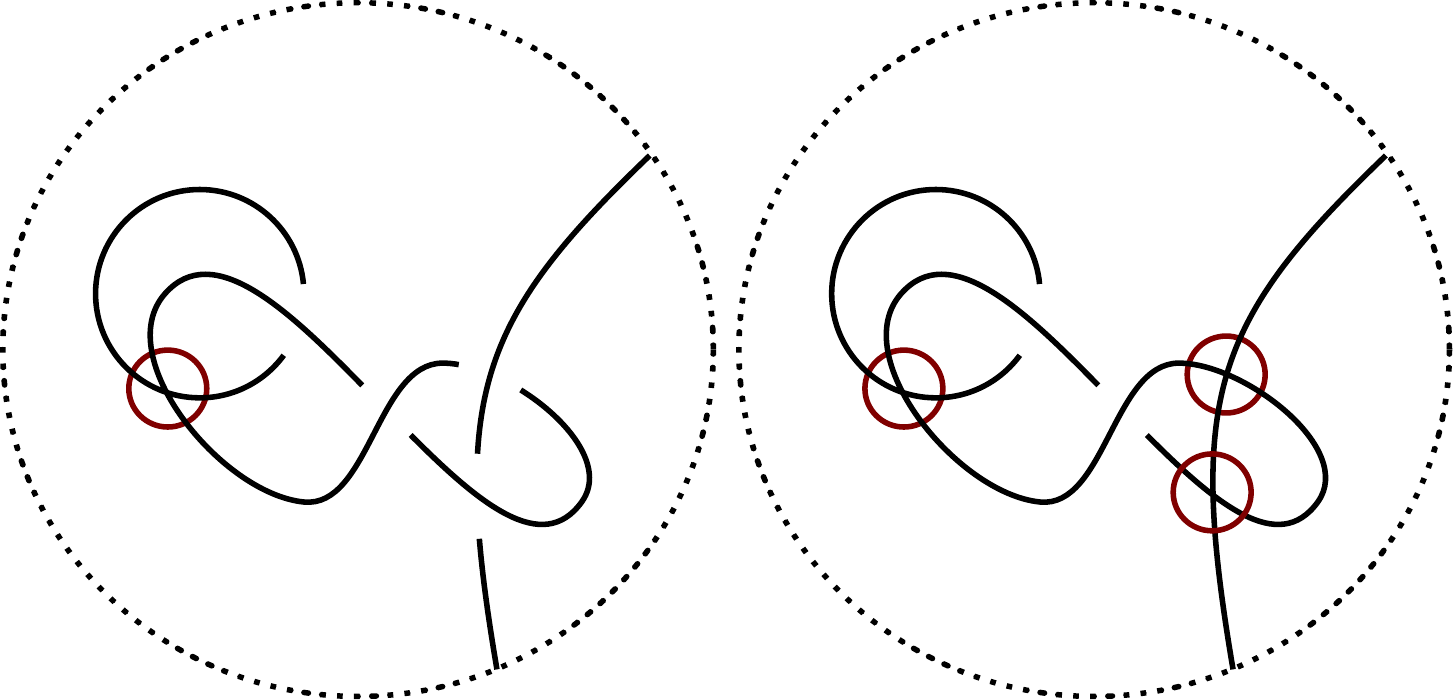}
  \caption{The left v-tangle diagram is not fully internal, but the right diagram has a fully internal component (the two internal v-circles).}
  \label{figure1-12}
\end{figure}

A virtual crossing is called \emph{negligible} if it is part of a fully internal part, e.g. all v-crossings of the right v-tangle diagram in Figure \ref{figure1-12} are negligible.

We call a v-tangle diagram $T^k_D$ \textit{nice} if there is a finite sequence of of vRM1, vRM2, vRM3 and mRM moves and virtualisations such that every v-circle is negligible, e.g. every v-link diagram is nice and every c-tangle diagram is nice. An example of a not nice v-tangle diagram is pictured in Figure \ref{figure2-4}.

We note that for a v-tangle diagram $T^k_D$ the chain complexes $\bn{T^k_D}^*$ and $\bn{T^k_D}_*$ are almost the same, i.e. they have the same vertices but possible different edges (which are still in the same position). The next lemma concretise the observation. The numbering in the lemma should be the same for the two complexes.

\begin{lemm}\label{lemma-factorbyisos}
Let $T^k_D$ be a v-tangle diagram. Let ${}_iS(\mathrm{in})^*$ and ${}_iS(\mathrm{in})_*$ denote the numbered saddles of $\bn{T^k_D}^*$ and of $\bn{T^k_D}_*$. Then we have a factorisation of the form ${}_iS(\mathrm{in})^*=\alpha\circ {}_iS(\mathrm{in})_*\circ\beta$ for invertible cobordisms $\alpha,\beta$ for every $i$ if $T^k_D$ is a nice v-tangle diagram.
\end{lemm}

\begin{proof}
That the saddles are topological equivalent is clear. We only have take a look at the decorations. 

The main point is the following observation: of the four outer (two on both sides) cobordisms in the bottom row of Figure \ref{figure1-2}, i.e. $\Id(1)^+_+$, $\Phi(1)^-_+$, $\Id(0)$ and $\Id(-1)^+_+$, only the third is not invertible. The first is the identity, the second and fourth are their own inverses. The third is not invertible because the $0$-indicator can not be changed to an $\pm 1$-indicator.

Note that none of the vRM1, vRM2, vRM3 and mRM moves nor a virtualisations changes the indicator of a saddle cobordism. Hence, it is sufficient to show the statement for a v-tangle diagram with only negligible v-crossings. From the observation above it is enough to show that every saddle gets an $0$-indicator in one closure iff it gets an $0$-indicator in the other closure.

The only possible way that a saddle gets an indicator from $\{+1,-1\}$ for one closure and an $0$-indicator for the other closure is the rightmost case in Figure \ref{figure0-4}. But for this case the existence of a non negligible v-crossing is necessary. Hence, we get the statement.    
\end{proof}

\begin{satz}\label{prop-nicechainiso}
Let $T^k_D$ be a v-tangle diagram. Then $\bn{T^k_D}^*$ and $\bn{T^k_D}_*$ are chain isomorphic if $T^k_D$ is nice.
\end{satz}

\begin{proof}
%If $T^k_D$ is not nice then we can use Lemma \ref{lemma-factorbyisos} to observe that every map of chain complexes $f\colon\bn{T^k_D}^*\to\bn{T^k_D}_*$ has a cobordism with an $0$-indicator. Hence, $f$ could not be an isomorphism.
%
Let $T^k_D$ be a nice v-tangle diagram. Then Lemma \ref{lemma-factorbyisos} ensures that every saddle is the same, up to isomorphisms, in $\bn{T^k_D}^*$ and $\bn{T^k_D}_*$. Furthermore Lemma \ref{lemma-everythingfine} ensures that both are well-defined chain complexes. Hence, the number of signs of every face is odd (also counting the ones from the decorations).

Hence, we can use a spanning tree argument to construct the chain isomorphism explicit, i.e. start at a leaf of a spanning tree of the Khovanov cube and change the orientations of the vertex such that the unique outgoing edge of the tree has the same sign in both cases (Lemma \ref{lemma-everythingfine} ensures that nothing changes modulo chain isomorphisms). Continue along the vertices of the spanning tree remove already visited leafs. This construction generates a chain isomorphism.

Next repeat the whole process, but change the indicators now. The Lemma \ref{lemma-factorbyisos} ensures that this will also generate a chain isomorphism. The chain isomorphism that we need is the composition.

See for example Figure \ref{figure1-10}. 
\end{proof}

The following theorem ensures that the choice of the position of the marker is not important for classical v-tangles or v-links. Hence, we can say that our notions extents both, the classical version and the version for v-tangle diagrams. Note that this is also an invariant for the question if a v-tangle diagram is classical or not.

\begin{theo}[Two different chain complexes]\label{theo-marker}
Let $T^k_D$ be a v-tangle diagram with two different *-marker positions. Let $\bn{T^k_D}^*$ and $\bn{T^k_D}_*$ be the geometric complex from Definition \ref{defin-geocomplex} for the two positions. Then the two complexes are equal in $\ukob(k)_{hl}$ if the v-tangle has $k=0$ or is a c-tangle.
\end{theo}

\begin{proof}
For a v-tangle with $k=0$ we can use the Proposition \ref{prop-nicechainiso} above.

For a classical v-tangle we can choose a diagram without virtual crossing without changing anything up to chain homotopies, because of Theorem \ref{theo-geoinvarianz}. Then we can use the Proposition \ref{prop-nicechainiso} again.
\end{proof}

Note that the whole construction can be done with an arbitrary \emph{non virtual closure} of a v-tangle diagram, i.e. capping it of without creating new v-crossings. One can repeat the arguments above to prove analogous statements like above for this construction, too.

\begin{bem}\label{rem-gradings}
We could use the Euler characteristic to introduce the structure of a graded category (see \cite{bn1}) on $\ucob_R(k)$ (and hence on $\kob_R(k)$).

The differentials in the geometric complex from Definition \ref{defin-geocomplex} have all $\deg=0$ (after a grade shift), because their Euler-characteristic is $-1$.

Then it is easy to proof that the geometric complex is a v-tangle invariant under graded homotopy. The proof of this statement is analogue to the one of D.~Bar-Natan from \cite{bn1}.
\end{bem}

\section{Circuit algebras and virtual tangles}\label{circuit}
In the present section we describe the notion of a \textit{circuit algebra}. A circuit algebra is almost the same as a planar algebra, but we allow virtual crossings.

Planar algebras were introduced by V.~Jones (see \cite{vj}) and were for example studied by D.~Bar-Natan in the case of classical Khovanov homology (see \cite{bn1}). Hence, we can use most of his constructions in our case, too.

A crucial difference is that we need to \textit{decorate} our \textit{circuit diagrams}. This is necessary because our cobordisms are also decorated.

We start the section with the definition of a circuit diagram. Recall that we call v-tangle diagram \emph{decorated} if all of its strings/circles are oriented. We call a cobordisms \emph{decorated} if it has gluing numbers at the boundary and indicators, see Definitions \ref{defin-vtangle} and \ref{defin-category}.

Moreover, in the whole section every v-tangle diagram should carry a *-marker.

\begin{defin}\label{defin-circuit diagram}
Let $D^2_o$ denote a disk embedded into $\real^2$, the so-called \emph{outside disk}. Let $\mathrm{I}_k$ denote $k\in\{0,\dots,m-1\}$ disks $D^2$ embedded into $\real^2$ such that for all $k\in\{0,\dots,m-1\}$ the disk $\mathrm{I}_k$ is also embedded into $D^2_o$ without touching the boundary of $D^2_o$, i.e. $\mathrm{I}_k\subset D^2_o\subset \real^2$ and $\mathrm{I}_k\cap\partial D^2_o=\emptyset$. We denote $\mathcal D_m=D^2_o-(\mathrm{I}_0\cup\cdots\cup\mathrm{I}_{m-1})$.

A \textit{circuit diagram with $m$ input disks} $\mathcal{CD}_m$ is a planar graph embedded into $\mathcal D_m$ with only vertices of valency one and four in such a way that every vertex of valency one is in $\partial\mathcal D_m$ and every vertex of valency four is in $\interior (\mathcal D_m)$. All vertices of valency four are marked with a virtual crossing. Again we allow circles, i.e. closed edges without any vertices. See for example Figure \ref{figure2-1}. A \textit{*-marked circuit diagram} is the same with $m+1$ extra *-marker for every boundary of $\mathcal D_m$. Moreover, we call the vertices at $\partial D^2_o$ the \textit{outer boundary points}.

A \textit{closure} of a *-marked circuit diagram with $m$ input disks $\mathrm{Cl}(\mathcal{CD}_m)$ is a circuit diagram with $m$ input disks and without any outer boundary points which is constructed from $\mathcal{CD}_m$ by capping of neighbouring strings starting from the *-marker and proceeding counter-clockwise.  

Note that we only cap of the outside disk and not the small inside disks.

A \textit{decoration} for a *-marked circuit diagram is a tuple of a numbering and an orientation of the strings of the diagram in such a way that its also an numbering and orientation of the closure. We call a circuit diagram together with a decoration a \textit{decorated circuit diagram}. See for example Figure \ref{figure2-1}. The decoration of the circuit diagram in this figure is also a decoration for the closure (the diagram together with the green lines).
\begin{figure}[ht]
  \centering
     \includegraphics[scale=0.6]{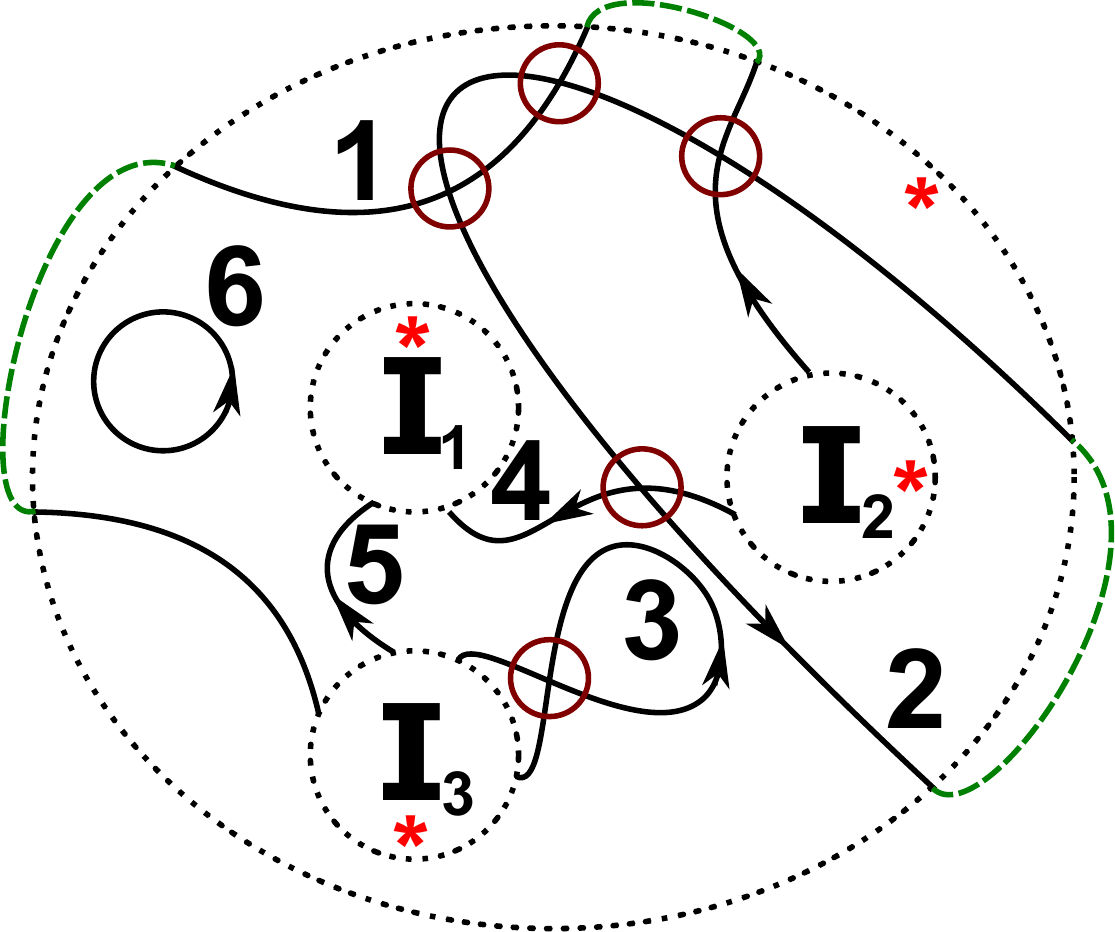}
  \caption{A decorated circuit diagram with three input disks.}
  \label{figure2-1}
\end{figure}
\end{defin}

With this setting we can implement the definition of a \textit{(decorated) circuit algebra}. We follow the definitions of D.~Bar-Natan and V.~Jones (see \cite{bn1}, \cite{vj}). Recall that our v-tangle diagrams should always be oriented.

\begin{defin}[Circuit algebra]\label{defin-circuit algebra}
Let $\mathrm{T'}(k)$ be the set of (decorated) v-tangle diagrams with $k$ boundary points and a *-marker and let $\mathrm{T}(k)$ denote the quotient by boundary preserving isotopies and generalised Reidemeister moves.

Furthermore let $\mathcal{CD}_m$ denote a (decorated) circuit diagram with $m$ input disks and $k'$ outer boundary points in such a way that the $j$-th input disk has $k_j$ numbered boundary points. 

Because $\mathcal{CD}_m$ has no classical crossings, this induces operations
\[
\mathcal{CD}_m\colon \mathrm{T'}(k_0)\times\cdots\times \mathrm{T'}(k_{m-1})\to \mathrm{T'}(k')\text{ and }\mathcal{CD}_m\colon \mathrm{T}(k_0)\times\cdots\times \mathrm{T}(k_{m-1})\to \mathrm{T}(k')
\]
by placing the $i$-th v-tangle diagram from $\mathrm{T^{(')}}(k_i)$ in the $i$-th boundary component of $\mathcal{CD}_m$, i.e. glue the v-tangle inside in such a way that the *-markers match. See the right side of Figure \ref{figure2-2}. 
\begin{figure}[ht]
  \centering
     \includegraphics[scale=0.6]{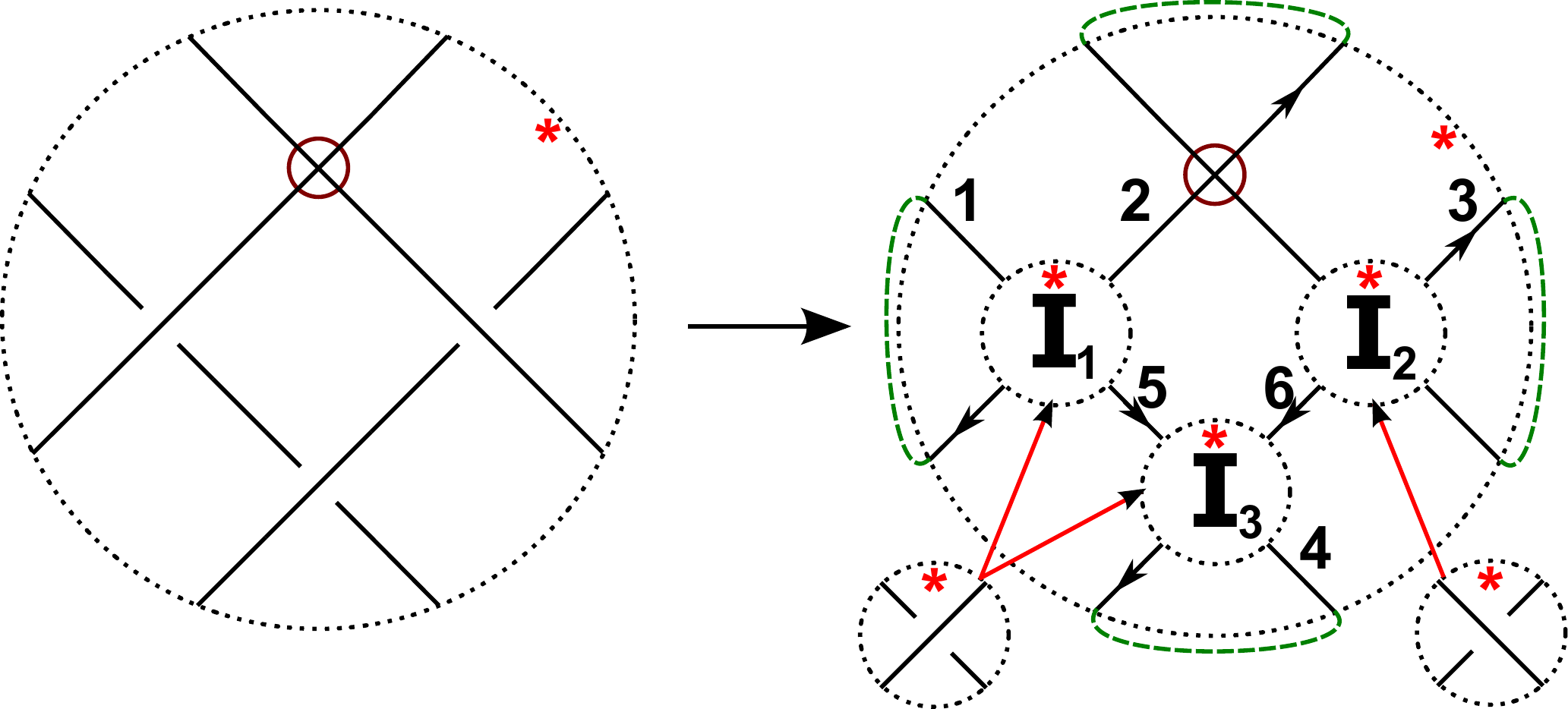}
  \caption{A decorated circuit diagram induced by a v-tangle.}
  \label{figure2-2}
\end{figure}

There is an identity operation on $\mathrm{T^{(')}}(k)$ (it is of the form $\jpg{5mm}{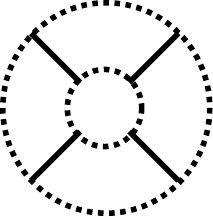}$) and the operations are compatible in a natural way (''associative'').

We call a set of sets $\mathfrak C(\omega)$ along with operations $\mathcal{CD}_m$ like above a \textit{circuit algebra}, provided that the identity and associativity from above holds.

If the operators and elements are decorated, first with numbers and orientations and latter only with orientations, then we call a set of sets $\mathfrak C(\omega)$ like above a \textit{decorated circuit algebra}. Note that in this case we have to define how the decorations change after gueing, since we can run into ambiguities.
\end{defin}

A necessary definition is the notion of a $X$-morphism. Here $X$ is a subset of a circuit algebra $\mathfrak C(\omega)$.

We need this notion because, in contrast to the classical case, the bracket is not a morphism of circuit algebras. See (c) of Theorem \ref{theo-semiloc} below. 

\begin{defin}[Circuit morphism]\label{defin-cirmor}
A \textit{morphism $\rho$ of (decorated) circuit algebras} $\mathfrak C(\omega)$ and $\mathfrak C'(\omega)$ is a map $\rho\colon \mathfrak C(\omega)\to \mathfrak C'(\omega)$ which satisfies $\rho\circ \mathcal{CD}_m=\mathcal{CD}_m\circ(\rho\times\cdots\times\rho)$ for every (decorated) operation $\mathcal{CD}_m$.

Let $X$ be a subset of $\mathfrak C(\omega)$. A \textit{$X$-morphism $\rho$ of (decorated) circuit algebras} $\mathfrak C(\omega)$ and $\mathfrak C'(\omega)$ is a map $\rho\colon \mathfrak C(\omega)\to \mathfrak C'(\omega)$ which satisfies
\[
\rho\circ \mathcal{CD}_m|_{X}=\mathcal{CD}_m\circ(\rho|_{X}\times\cdots\times\rho|_{X})
\]
for every (decorated) operation $\mathcal{CD}_m$
\end{defin}

Here are some examples. The reader may also check the corresponding section in the paper of D.~Bar-Natan from \cite{bn1}.
\begin{beisp}\label{beisp-circuit}
The first example is the set $\Ob(\ucob(\omega^{*}))$ from Definition \ref{defin-category}, i.e. v-tangles diagrams with $k\in\nat$ boundary points without classical crossings and an extra *-marker. This is a sub-circuit algebra of the circuit algebra that allows classical crossings.

This example gets more exciting if we want to view it as a decorated circuit algebra. We have to define the operations, i.e. the gluing inside the input disks, in more detail now, since we can run into ambiguities, see top row of Figure \ref{figure2-3}.
\begin{figure}[ht]
  \centering
     \includegraphics[scale=0.7]{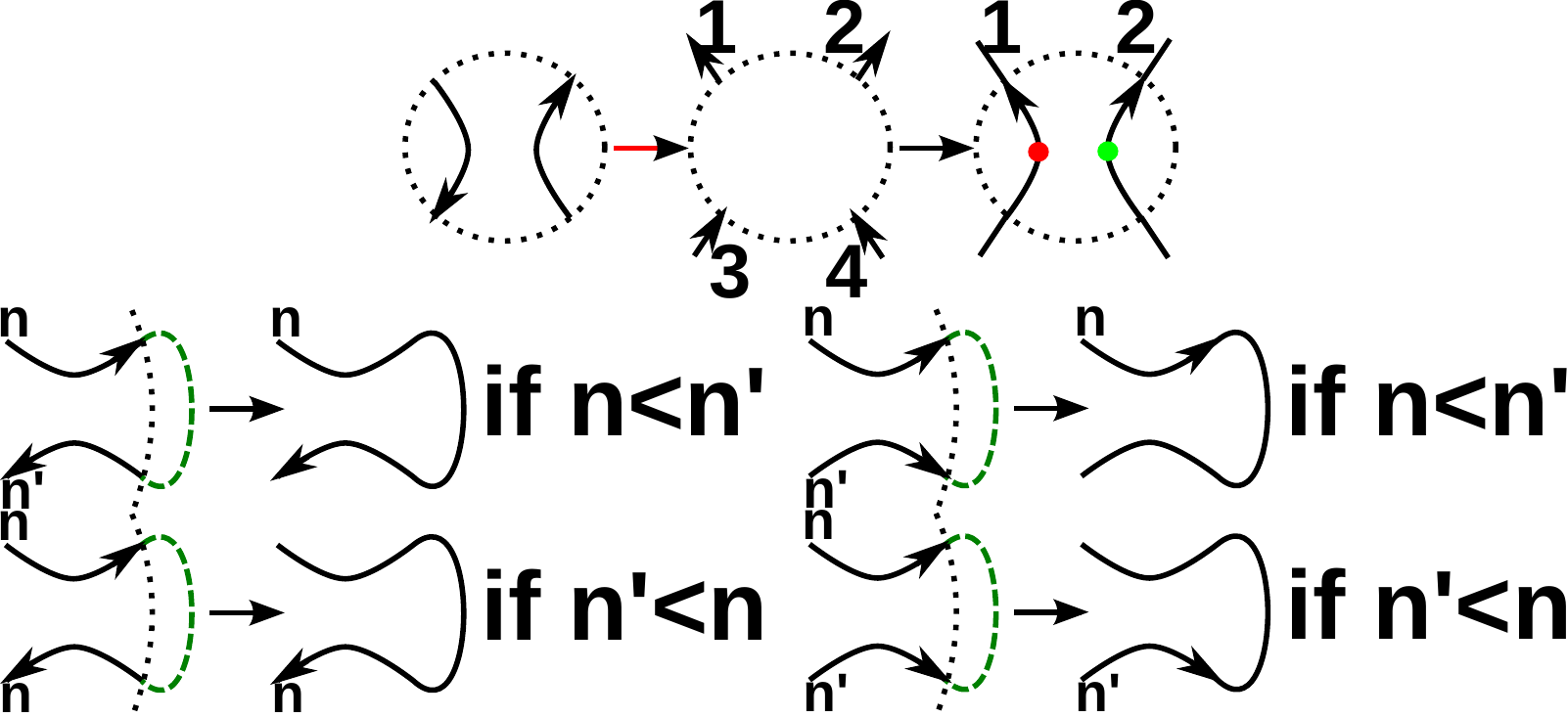}
  \caption{The operation in the decorated circuit algebra.}
  \label{figure2-3}
\end{figure}

We can run into ambiguities if the decorations of the operator get glued together do not match up. In this case we define the new decoration based on the rule \textit{``lower first''}, i.e. the new number should be the lower and the new orientation should be the one from the lower numbered string. See lower row of Figure \ref{figure2-3}. Not all four cases are pictured but it should be clear how the other two work.  

Furthermore if we glue a decorated v-tangle diagram in an input disk, then we run into ambiguities if the shared decorations, i.e. the orientations, does not match up. In this case we change the decorations of the v-tangle diagram. We add in a \textit{red dot} if we have to change the orientation and a \textit{green dot} otherwise. This is pictured in the top row of Figure \ref{figure2-3}.

The reader should check that this gives rise to a decorated circuit algebra.

Another important example is the whole collection $\Mor(\ucob(\omega^{*}))$ from Definition \ref{defin-category}, i.e. decorated cobordisms with $k\in\nat$ vertical boundary lines and an extra *-marker. We want to view this example as a decorated circuit algebra again.

Hence, we have to define the operations. The most interesting point are the decorations again, because it should be clear how to glue a cobordism with $m$ vertical boundary lines into $\mathcal{CD}_m\times [-1,1]$. This time we have to define the behaviour of three decorations, i.e. the gluing numbers and the signs and indicators.

The gluing numbers should be treated like the orientations before, i.e. if they do not match, then we use the gluing number of the lower numbered string. The indicators (recall that they are just numbers in $\{0,+1,-1\}$) should be multiplied. Recall that a cobordism with an $0$-indicator does not get any gluing numbers. We simply remove them in this case.

We get the following interaction of the operation of $\mathcal{CD}_m$ on $\Ob(\ucob(\omega^{*}))$ and $\Mor(\ucob(\omega^{*}))$ (compare to Figure \ref{figure2-3}):
\begin{itemize}
\item A saddle is composed with $\Phi^-_+$ iff the v-tangle gets a red dot at the corresponding position.
\item A saddle is composed with an $0$-indicator surfaces iff the v-tangle gets a red and a green dot in both resolutions at the corresponding position.
\item A saddle is composed with an $-1$-indicator surfaces iff the v-tangle numbers get identified at the top resolution at the corresponding position.
\item A saddle gets an extra minus sign iff the v-tangle numbers change in such a way that the sign of $n'-n$ is changed at the corresponding position.
\end{itemize} 

These rules define a new decoration for the new cobordism. The reader should check again that this gives rise to a decorated circuit algebra.
\end{beisp}

We summarise the notions in a definition. Recall that v-tangle diagrams are decorated with orientations and a *-marker and cobordisms are decorated with orientations on the boundary, an indicator and a *-marker.

\begin{defin}[Dot-calculus]\label{defin-dotcalculus}
Let $\mathcal{CD}_m$ denote a decorated circuit diagram with $m$ input disks and $k'$ outer boundary points in such a way that the $j$-th input disk has $k_j$ numbered boundary points.

Then $\mathcal{CD}_m$ induces an associative and unital (in the above sense) operation on decorated v-tangles (with a corresponding number of boundary points) by the ``lower first''-rule, i.e. if the orientation does not match, then the lower number induces the new orientation. Put a red dot on every string that has its orientation changed and a green dot otherwise. We call this the \emph{tangle dot-calculus}.

Moreover, $\mathcal{CD}_m$ induces an associative and unital (in the above sense) operation on decorated cobordisms (with a corresponding number of boundary lines) by the ``lower first''-rule, i.e. if the orientation does not match, then the lower number induces the new orientation. Put a red dot on every string that has its orientation changed and a green dot otherwise and compose the corresponding boundary with $\Phi^-_+$ iff the string has a red dot, multiply indicators via identity surfaces with corresponding indicators $+1,-1$ iff the v-tangle numbers get identified at the bottom/top resolution at the corresponding position, multiply with an $0$-identity iff in both resolutions the strings are identified and multiply with a formal sign iff the numbers change in such a way that the sign is changed at the corresponding position. We call this the \emph{dot-calculus}.

If the crossings carries x-markers, then they should be in the same position as before after the operation.  
\end{defin}

The reader should compare the notions above with Figure \ref{figure0-main}.

We follow the notions of D.~Bar-Natan in the next theorem, i.e. we use the notions $\ukob_R(k)$ and $\ukob_R(k)_h$ from Definition \ref{defin-category2}, but we also allow $k=\omega^{*}$, i.e. *-marked v-tangles with a finite number of boundary points. We denote the collection of *-marked nice v-tangles as $\mathrm{NT^{*}}\subset\mathrm{T^{*}}$.

Recall that every v-link diagram and every c-tangle diagram is nice.

\begin{theo}[Semi-locality]\label{theo-semiloc}
\begin{itemize}
\item[(a)] The whole collection $\ukob_R(\omega^{*})$ has a natural structure of a decorated circuit algebra.
\item[(b)] The operations on $\ukob_R(\omega^{*})$ preserves homotopy equivalence. Thus we obtain that $\ukob_R(\omega^{*})_h$ has a natural structure of a decorated circuit algebra.
\item[(c)] The bracket $\bn{\cdot}\colon\mathrm{T^{*}}(\omega)\to\ukob_R(\omega^{*})$ is not a morphism of decorated circuit algebras.
\item[(d)] The bracket $\bn{\cdot}\colon\mathrm{T^{*}}(\omega)\to\ukob_R(\omega^{*})$ is a $\mathrm{NT^{*}}$-morphism of decorated circuit algebras.
\end{itemize}
\end{theo}

\begin{proof}
We follow the proof of the classical version in \cite{bn1}. Thus we start by extending the structure of a decorated circuit algebra of the set $\Ob(\ucob(\omega^{*}))$ and the one of $\Mor(\ucob(\omega^{*}))$ from Example \ref{beisp-circuit} to decorated circuit structures $\Ob(\mat(\ucob(\omega^{*})))$ and $\Mor(\mat(\ucob(\omega^{*})))$ multilinear.

Let $\mathcal{CD}_m$ denote a decorated circuit diagram with $m$ input disks and $k_0$ outer boundary points in such a way that the $j$-th input disk has $k_j$ numbered boundary points.

(a) For complexes $(\mathrm{Co}_j,d_j)$ with $j\in\{0,\dots,m-1\}$
\[
\begin{xy}
  \xymatrix{
      \cdots\ar[r]^{d^{l-1}_{j}} & \mathrm{Co}^l_j\ar[r]^{d^l_j}& \mathrm{Co}^{l+1}_{j}\ar[r]^{\phantom{a}d^{l+1}_{j}}& \cdots
  }
\end{xy}
\]
of v-tangle diagrams with $k_j$ boundary points we define the new ``tensored'' complex $(\mathrm{Co},d)=\mathcal{CD}_m(\mathrm{Co}_0,\dots,\mathrm{Co}_{m-1})$ as follows.

Let $\alpha_i,\beta_i$ denote the composition with the morphisms that we compose after applying the circuit diagram on cobordisms (see Example \ref{beisp-circuit} and Definition \ref{defin-dotcalculus} above), i.e. the red dots induce a composition with $\Phi^-_+$ (or a composition with a $0$-identity surface in the degenerated case) and a change in the numbering induces a composition with a cobordisms that changes signs and indicators.

Therefore we denote the operation of $\mathcal{CD}_m$ on cobordisms, i.e. the dot-calculus, as $\alpha\circ\mathcal{CD}_m\circ\beta$ to illustrate the difference to the classical case. We skip this notion for the objects.

The $i$-th chain module should be
\[
\mathrm{Co}^i=\bigoplus_{i=j_0+\dots+j_{m-1}}\mathcal{CD}_m(\mathrm{Co}_0^{j_0},\dots,\mathrm{Co}_{m-1}^{j_{m-1}})
\]
and the chain maps should be
\[
d|_{\mathcal{CD}_m(\mathrm{Co}_0^{j_0},\dots,\mathrm{Co}_{m-1}^{j_{m-1}})}=\sum_{i=0}^{m-1}\alpha\circ\mathcal{CD}_m(\mathrm{Id}_{\mathrm{Co}_0^{j_0}},\dots,d_i,\dots,\mathrm{Id}_{\mathrm{Co}_{m-1}^{j_{m-1}}})\circ\beta.
\]

Beware that the needed minus signs are not missing. They are already implied in the definition of the decorated circuit algebra structure of $\Mor(\mat(\ucob(\omega^{*})))$, i.e. the dot-calculus in particular adds or removes signs. See Example \ref{beisp-circuit} and Definition \ref{defin-dotcalculus}.

We use Lemma \ref{lemma-everythingfine} to see that we still get a chain complex.

Moreover, we need to define the operation of $\mathcal{CD}_m$ on the morphisms of the category $\ukob_R(\omega^{*})$, i.e. on the chain maps.

So let $f\colon(\mathrm{Co}_1,d^1_*)\to(\mathrm{Co}_2,d^2_*)$ be a chain map. Every chain map itself is a sequence of matrices of decorated cobordisms $f=(\dots,f_i,f_{i+1},\dots)$. The operation on the chain map should just insert the cobordisms topological, i.e. without changing gluing numbers, indicators or signs. Note that we do not run into ambiguities since the decorations are defined for the closure for the fixed *-marker.

We denote an entry of the matrix $f_i$ with abuse of notation as $f_i$ and similar for $d^{1,2}$. Hence, we have to check that the commutativity of the faces
\[
\xymatrix{
\cdots\ar[r] & \mathrm{Co}_1^{i}\ar[r]^{d_1^{i}}\ar[d]^{f_i} & \mathrm{Co}_1^{i+1}\ar[r]\ar[d]^{f_{i+1}} & \cdots\\
\cdots\ar[r] & \mathrm{Co}_2^{i}\ar[r]^{d_2^{i}} & \mathrm{Co}_2^{i+1}\ar[r] & \cdots.
}
\]
is preserved. The operation on the $d^*_{1,2}$ is defined as
\[
\mathcal{CD}_m(d^i_1)=\alpha_1^{i+1}\circ d^i_1\circ\beta_1^{i}\;\;\text{and}\;\;\mathcal{CD}_m(d^i_2)=\alpha_2^{i+1}\circ d^i_2\circ\beta_2^{i},
\]
where $\alpha^*_*,\beta^*_*$ are the decoration, indicator and signs changing morphisms.

The rest is a case-by-case check using the relations in Lemma \ref{lemm-calc} and the equation $f_{i+1}\circ d_1^i=d_2^{i+1}\circ f_i$.
\begin{itemize}
\item If one of the two sides of the equation $f_{i+1}\circ d_1^i=d_2^{i+1}\circ f_i$ has a $0$-indicator, then the other has also a $0$-indicator. Hence, in this case everything works out by the commutation relations in Lemma \ref{lemm-calc}.
\item Since everything works out topological and the indicator changes are spread if strings get glued together, we see that if $\alpha^*_*,\beta^*_*$ are only indicator changes, then, by commutativity of the indicator changes, the equation is preserved.
\item An analogous argument works if $\alpha^*_*,\beta^*_*$ are only sign changes.
\item Hence, lets assume that $\alpha^*_*,\beta^*_*$ are only decoration changes, i.e. tensors of $\Phi^-_+$ and $\Id^+_+$, and we do not have any $0$-indicators. Then one has different cases to check, i.e. based on the indicators for $d_1,d_2$ and $f$. They all work analogous, so we only do one case, i.e. only $d_1^i$ and $f_i$ should have negative indicators. Then we have by the commutation relations
\begin{align*}
f_{i+1}\circ \mathcal{CD}_m(d^i_1) &= f_{i+1}\circ \alpha_1^{i+1}\circ d^i_1\circ\beta_1^{i}
= f_{i+1}\circ d^i_1\circ \bar{\alpha}_1^{i+1}\circ\beta_1^{i}\\
&= -d^{i+1}_2\circ f_{i}\circ \bar{\alpha}_1^{i+1}\circ\beta_1^{i}
= \alpha_1^{i+1}\circ d^{i+1}_2\circ f_{i}\circ \beta_1^{i},
\end{align*}
where $\bar{\alpha}$ should denote a short hand notation for the one from Lemma \ref{lemm-calc}. Note that the equation $f_{i+1}\circ d_1^i=d_2^{i+1}\circ f_i$ also holds for indicators, i.e. not all combinations of indicators for $d_1,d_2$ and $f$ are possible.
\end{itemize}

(b) The proof of (b) is an analogous argument like in (a), i.e. we define the operation on the homotopies $h_i$ like before and we get the statement with an argument like before and the corresponding statement for tensor products.

(c) This is true, because a surfaces with an $0$-indicator can not be changed to a surfaces with an $\pm 1$-indicator. For an concrete example see Figure \ref{figure2-4}. The upper complex is not homotopy equivalent to the lower complex because we can not change the $0$-indicator. 
\begin{figure}[ht]
  \centering
     \includegraphics[scale=0.7]{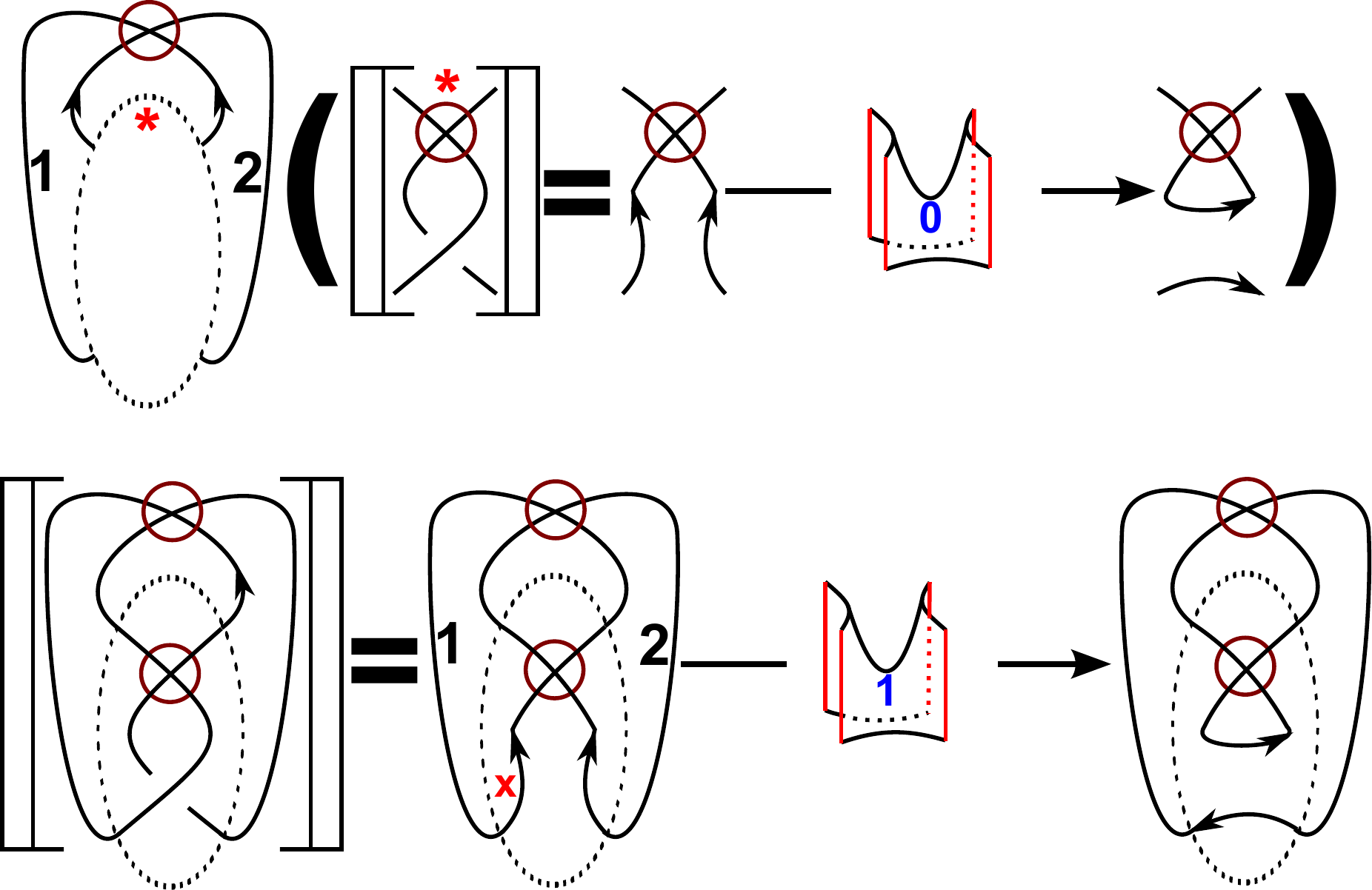}
  \caption{A counterexample. The diagram is not a nice v-tangle diagram.}
  \label{figure2-4}
\end{figure}

(d) This follows because of Lemma \ref{lemma-factorbyisos}, i.e. if a surface with an $0$-indicator appears in the complex before the operation $\mathcal{CD}_m$, then it will appear for both closures. Otherwise the, i.e. if the indicators are $+1,-1$, then we can use the factorisation of Lemma \ref{lemma-factorbyisos}. Note that the proof of Lemma \ref{lemma-factorbyisos} does not depend on Theorem \ref{theo-geoinvarianz}.

Then the equality is a consequence of the definition of the operation on the cobordisms and of an argument like in \cite{bn1}.
\end{proof} 

Note that the next corollary finish the proof of Theorem \ref{theo-geoinvarianz}.

\begin{koro}\label{koro-reidemeister}
Let $R_D^{2,4,6}$ be a v-tangle diagram that looks like a generalised Reidemeister move (see Figure \ref{figure0-1}). Then 
\[
\bn{\mathcal{CD}_m(T_D^{k_0}\times\dots\times T_D^{k_{m-1}})}=\mathcal{CD}_m(\bn{T_D^{k_0}}\times\dots\times\bn{T_D^{k_{m-1}}})
\]
for any operation $\mathcal{CD}_m$ if $T_D^{k_i}$ is either a classical tangle or a Reidemeister-tangle $R_D^{2,4,6}$ for $i\in \{0,\dots,m-1\}$. 
\end{koro}

\begin{proof}
This is a direct consequence of point (d) of Theorem \ref{theo-semiloc} and the observation that the Reidemeister-tangles, except the mRM-tangle, are all part of the set $\mathrm{NT^{*}}$.

That the statement also works if one of the $T_D^{k_i}$ is a mRM-tangle is a case-by-case check for the two different closures of a mRM-tangle (both sides), i.e. in all cases the appearing cobordisms have $+1,-1$ indicators.
\end{proof}

\begin{bem}\label{rem-gradings2}
Like before and like in the paper of D.~Bar-Natan \cite{bn1}, the whole construction works with gradings, too.

But since this is a straightforward generalisations of the notions in \cite{bn1}, we skip the details here. 
\end{bem}

\section{An application: Degeneration of Lee's variant}\label{leedeg}
In the whole section let $R$ denote a commutative and unital ring such that $2$ is invertible, e.g. $R=\zet\left[\frac{1}{2}\right]$. Moreover, through the whole section we denote the geometric complex as $\bn{-}$ and a corresponding algebraic version, i.e. an uTQFT (see \cite{tub}) applied to the geometric complex, as $\mathcal F(\bn{-})$ or short as $\mathcal F(-)$, e.g. we denote Lee's version as
\[
\mathcal F_{\mathrm{Lee}}(-)=\mathcal F(\bn{-}_{\mathrm{Lee}}).
\]
Note that, in order for the signs to work out correct, we have to fix a marker position. In the whole section we, by convention, say that the marker for $\slashoverback$ is at the left side and for $\backoverslash$ is at the top.

Recall (see \cite{lee} for the classical and \cite{tub} for the virtual case) that \emph{Lee's variant}, i.e. the algebraic version, for v-links is given by the following maps and the filtered module $A=A_{\mathrm{Lee}}=R[X]/(X^2=1)$.
\[
m^{++}_+\colon A\otimes A\to A,\;\begin{cases}1\otimes 1\mapsto 1, & X\otimes X\mapsto 1 \\ 1\otimes X\mapsto X, & X\otimes 1\mapsto X\end{cases},\;\Phi^-_+\colon A\to A,\; 1\mapsto 1,\;X\mapsto -X
\]
and
\[
\Delta_{++}^+\colon A\to A\otimes A,\;\begin{cases}1\mapsto 1\otimes X+X\otimes 1\\ X\mapsto 1\otimes 1+X\otimes X\end{cases},\;\theta\colon A\to A,\; 1\mapsto 0,\;X\mapsto 0.
\]
Moreover, recall that the geometric picture of Lee's variant is given by the \emph{dot-relations} in Figure \ref{figure3-1} with $t=1$, while the graded case of the Khovanov complex is $t=0$. Note that $\frac{1}{2}\in R$ allows us to use the \emph{dot-relation} in Figure \ref{figure3-1} instead of the local relations of Figure \ref{figure0-2}. We give an example of the Lee complex of a v-knot in Example \ref{beisp-leedeg}.

\begin{figure}[ht]
  \centering
     \includegraphics[width=0.9\linewidth]{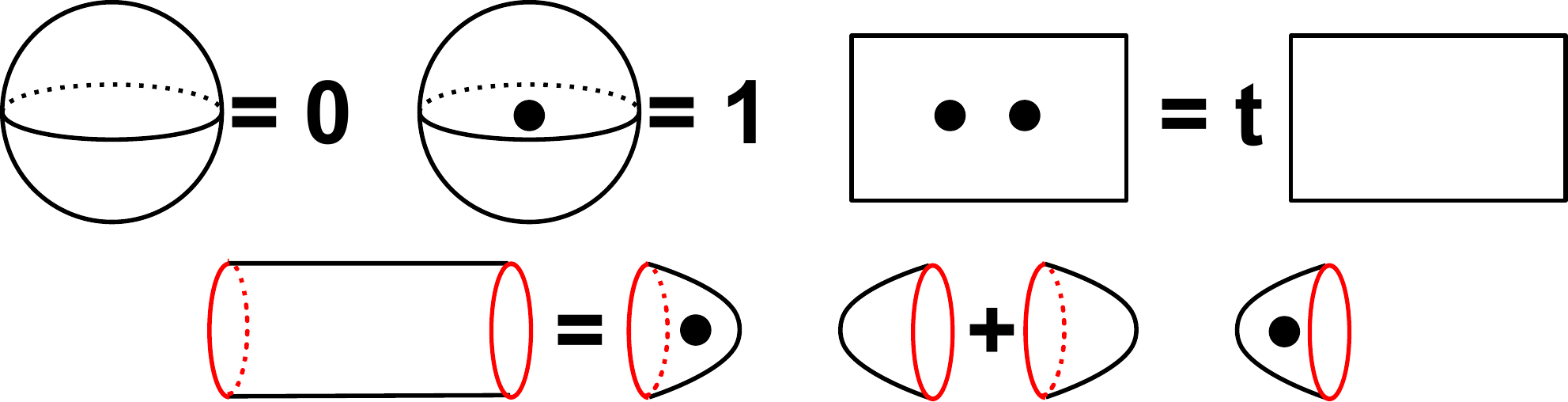}
  \caption{The dot-relations.}
  \label{figure3-1}
\end{figure}

Lee's variant has a remarkable property in the classical case, i.e. E.S.~Lee showed that her variant just ``counts'' the number of components of the c-link, i.e. she showed that (for $R=\rat$) the homology of an $n$-component link $L$ is
\[
H(\mathcal F_{\mathrm{Lee}}(L))\iso\bigoplus_{2^n}\rat.
\]
So on the first hand this seems to be a boring invariant. But J.~Rasmussen (see \cite{ra}) used this degeneration in a masterfully way to define the \emph{Rasmussen invariant} of a c-knot.

Therefore a natural question is if this degeneration of Lee's variant is still true for v-links. In this section we show that this is indeed the case.

Note that this is a unexpected result since $\theta=0$ for $2^{-1}\in R$ (see the relations in Definition \ref{defin-category}). Hence, there are ``tons'' of $0$-morphisms in the complex. But these $0$-morphisms also come with isomorphism ``in a lot'' cases. The following example for the Lee complex of a v-knot is a blueprint of this effect.

\begin{beisp}\label{beisp-leedeg}
Consider the diagram of the virtual trefoil $L_D$ given in Figure \ref{figure4-1}. In this example the number of negative crossings is zero, i.e. the leftmost object is the $0$-degree of the chain complex.
\begin{figure}[ht]
  \centering
     \includegraphics[width=0.9\linewidth]{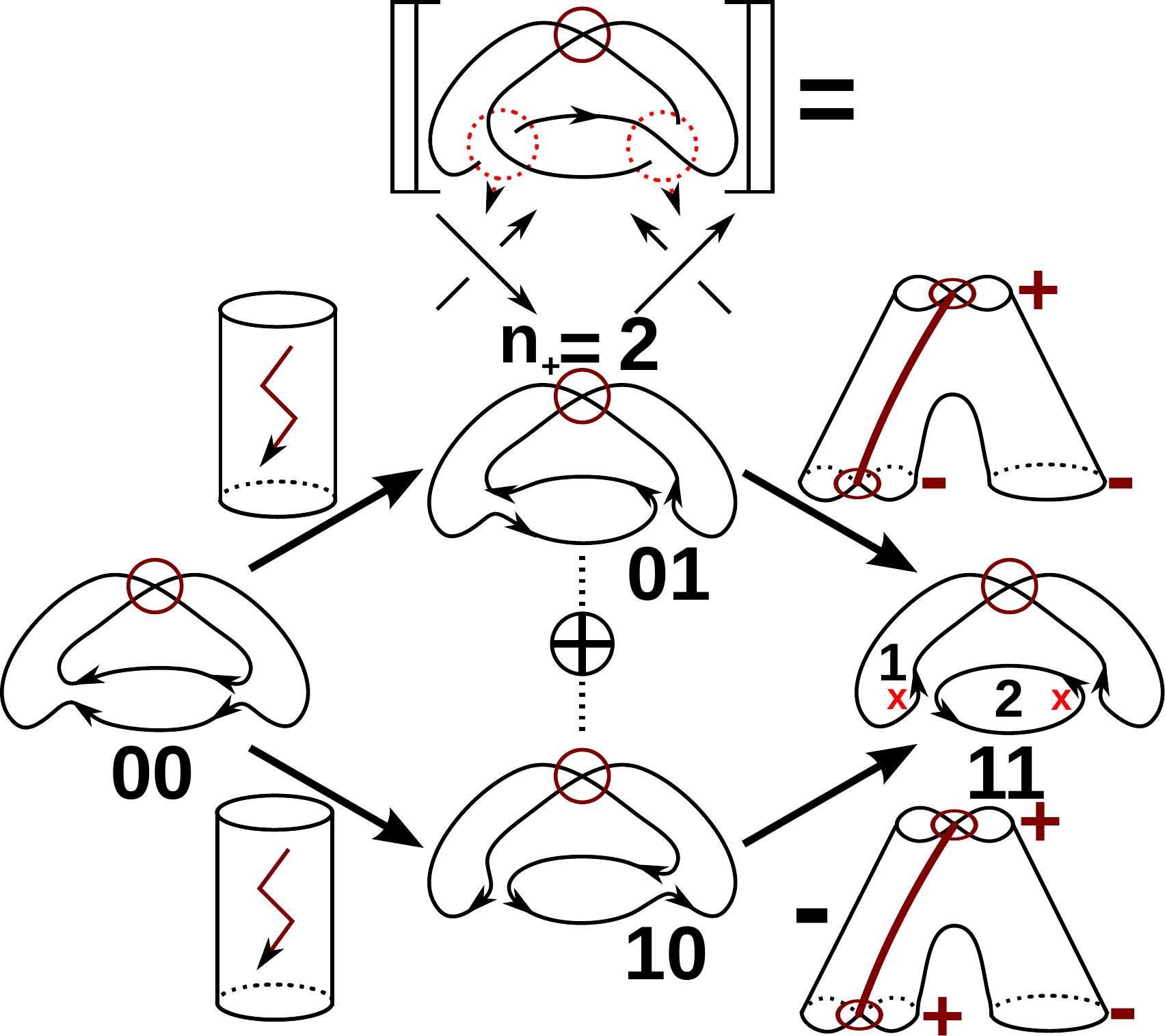}
  \caption{The Lee complex of the v-trefoil. Note that the first map is a $0$-morphism, but the second is an isomorphism.}
  \label{figure4-1}
\end{figure}
Let us consider $R=\rat$. Then $\theta=0$ and therefore the first two maps are $0$-morphisms. But note that the two right morphisms are not the same, i.e. one is $\Delta^+_{--}$ and the other is $\Delta^+_{+-}$. So on the algebraic level we get, using the maps from before, the following complex if we fix $B_1=\{1,X\}$ as a basis for $A$ and $B_2=\{1\otimes 1,1\otimes X,X\otimes 1,X\otimes X\}$ for $A\otimes A$.
\[
\begin{xy}
  \xymatrix{
      A\ar[rrr]^{\begin{pmatrix}
      0 & 0 \\
      0 & 0 \\
      0 & 0 \\
      0 & 0 
      \end{pmatrix}} &  & & A\oplus A\ar[rrr]^{\begin{pmatrix}
      0 & 1 & 0 & -1\\
      -1 & 0 & -1 & 0\\
      -1 & 0 & 1 & 0\\
      0 & 1 & 0 & 1
      \end{pmatrix}} & & & A\otimes A.
      }
\end{xy}
\]
An easy calculations shows that the second matrix is an isomorphism. Hence, the homology of the virtual trefoil is only non-trivial for $k=0$, i.e.
\[
H_k(\mathcal F_{\mathrm{Lee}}(L_D))=\begin{cases}\rat\oplus\rat, & \text{if}\;k=0,\\ 0, &\text{else}.\end{cases}
\]

Another example is the v-knot in Figure \ref{figure4-2}, e.g. with the pictured orientation and numbering of the circles from left to right, the three outgoing morphisms from resolution $000$ to $001$, $010$ and $100$ are (up to, in this case, not important signs) the morphisms $m^{+-}_+$, $m^{++}_-$ and $m^{--}_+$, i.e. one alternating and two non-alternating. Hence, the kernel is trivial. The reader should check that the rest also works out in the same fashion.
\end{beisp}
The approach (we follow D.~Bar-Natan and S.~Morrison \cite{bnsm}) to show that the degeneration is still true is the following. First we define two orthogonal idempotents in our category, which we call \emph{down and up}. We can look at the \emph{Karoubi envelope} of our category, denoted $\KAR(\ukob_R(k))$.

The idea of the Karoubi envelope is to find a ``completion'' of a category such that every idempotent splits. It is named after the french mathematician M.~Karoubi, but it already appears in an earlier work by P.~Freyd in \cite{pf}.

Then we show that the geometric complex of a simple crossing (as a v-tangle), if considered in $\KAR(\ukob_R(k))$, is homotopy equivalent to a very simple complex with only $0$-morphisms. After that we use the local construction from Section \ref{circuit} to finish the proof.

Moreover, we proof the following interesting result of the number of decorations of v-link resolution with the ``colours'' down and up. Note that we call an oriented resolution $\mathrm{Re}$ of a v-link diagram \emph{non-alternating} if it is of the form \uu\,or \dd\, at the corresponding positions of the saddles. Recall that all the v-link diagrams should be oriented and that such a diagram with $n\in\nat_{>0}$ components has $2^n$ different orientations $\mathrm{Or}_1,\dots,\mathrm{Or}_{2^n}$.

We note that one can also colour the resolutions with honest colours, say red and green, in such a way that the colour changes at every v-crossing. We call this a \emph{colouring of a v-link resolution} if at the corresponding saddle-position the colours are different, i.e. (red,green) or (green,red). The reader should compare this with the coloured dots in Figure \ref{figure0-main}.
\begin{theo}[Non-alternating resolutions]\label{theo-nonalternating}
Let $L_D$ denote a v-link diagram with $n\in\nat_{>0}$ components. There are bijections of sets
\begin{align*}
\{\mathrm{Or}\mid \mathrm{Or}\;\text{is an orientation of}\;L_D\}&\simeq\{\mathrm{Re}\mid \mathrm{Re}\;\text{is a non-alternating resolution of}\;L_D\}\\
&\simeq\{\mathrm{Co}\mid \mathrm{Co}\,\text{is a coloured resolution of}\;L_D\}.
\end{align*}
If $L_D$ is a v-knot, i.e. $n=1$, then the two non-alternating resolutions are in homology degree $0$.
\end{theo}

\begin{proof}
With a slight abuse of notation let us denote the first two sets as $\mathrm{Or}$ and $\mathrm{Re}$. To show the existence of a bijection we construct an explicit map $f\colon\mathrm{Or}\to\mathrm{Re}$ and its inverse.

Given an orientation $\mathrm{Or}$ of the v-link diagram $L_D$, the map $f$ should assign the resolution $\mathrm{Re}$ which is obtained by replacing every oriented crossing of the form $\overcrossing$ and $\undercrossing$ to \uu. This is clearly an injection.

Now, given a non-alternating resolution $\mathrm{Re}$, we assign to it an orientation of $L_D$ in the following way. At any non-alternating part of the form \uu\,and \dd\,replace the non-alternating part with the corresponding oriented crossing $\overcrossing$ and $\undercrossing$ (or a rotation in the \dd\, case).

Note that both maps are well-defined (by definition of oriented v-link and oriented resolution).

These two maps are clearly inverses and for a v-knot the corresponding non-alternating resolutions are in homology degree $0$ since all $n_+$-crossings are resolved $0$ and all $n_-$-crossings are resolved $1$ in this procedure.   

To see the second bijection use a checker-board colouring of the v-link diagram. Then start at any point of the non-alternating resolution and use the right-hand rule, i.e. the index finger follows the orientation and the string should get the colour of the face on the side of the thumb.
\end{proof}
\begin{beisp}\label{beisp-nonalt}
Let $L_D$ be the v-knot diagram in Figure \ref{figure4-2}. Then only the $011$ resolution of the v-knot diagram allows a non-alternating resolution. Moreover, the orientation of the diagram induces this non-alternating resolution by replacing the three crossings with \dd,\,\uu\,and \uu. The other orientation induces the non-alternating resolution \uu,\,\dd\,and \dd. Note that, by construction, these resolution are in homology degree $0$. A computation like in Example \ref{beisp-leedeg} shows that these two non-alternating resolutions give the only two generators of the homology, i.e.
\[
H_k(\mathcal F_{\mathrm{Lee}}(L_D))=\begin{cases}\rat\oplus\rat, & \text{if}\;k=0,\\ 0, &\text{else}.\end{cases}
\]
\begin{figure}[ht]
  \centering
     \includegraphics[width=0.9\linewidth]{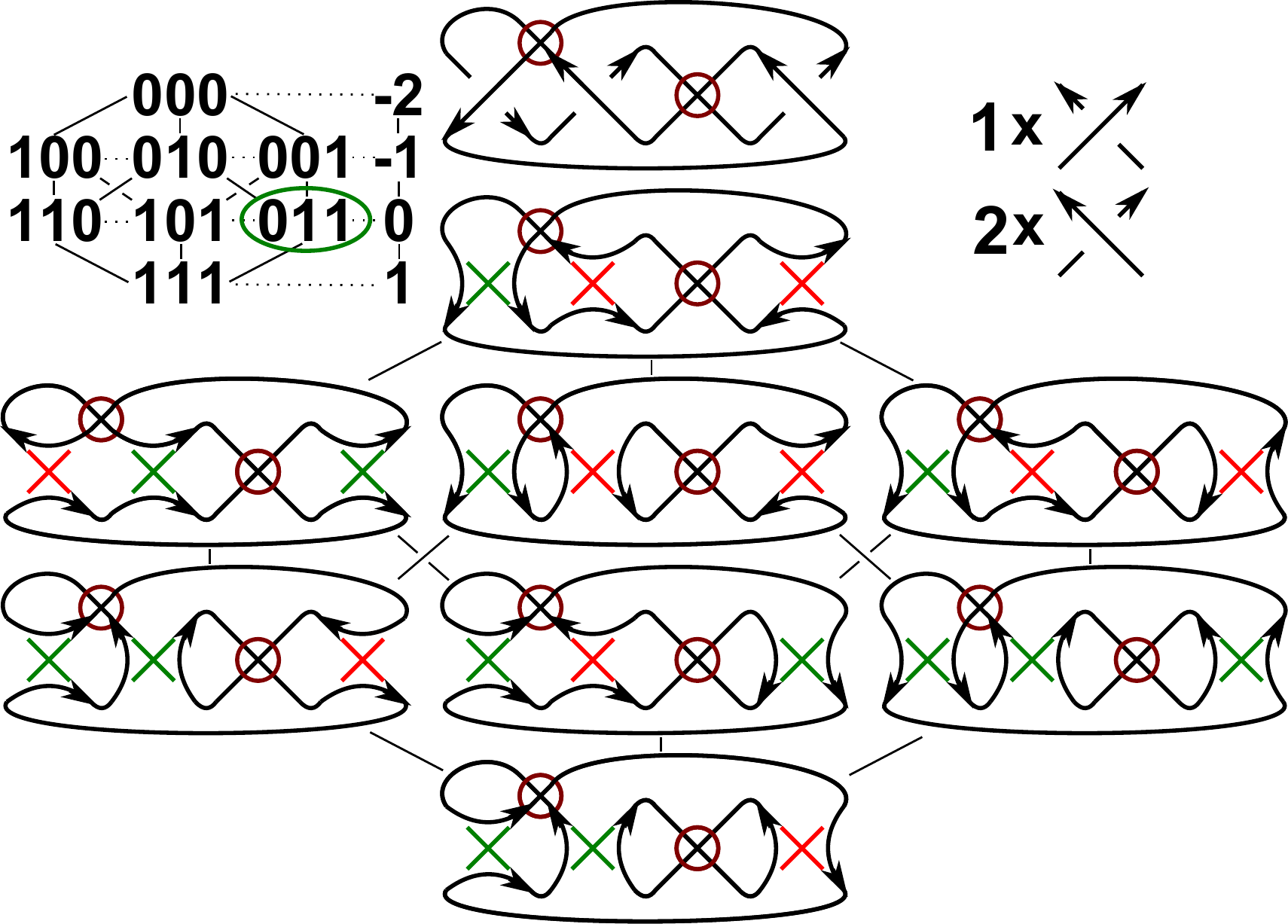}
  \caption{There are exactly two non-alternating resolutions, i.e. the one pictured and the one with all orientations reversed.}
  \label{figure4-2}
\end{figure}
\end{beisp}
We recall the motivation, definition and some basic properties of the \emph{Karoubi envelope} of a pre-additive category $\mathcal C$. We denote the envelope as $\KAR(\mathcal C)$.

For any category the notion of an idempotent morphisms, i.e. an arrow with $e\circ e=e$, makes sense. Moreover, in a pre-additive category the notion $\Id-e$ also makes sense. A classical trick in modern algebra is to use an idempotent, e.g. in an algebra $A$, to split the algebra into
\[
A\simeq eA\oplus (1-e)A.
\]
Hence, it is a natural question to ask if on can ``split'', given an idempotent $e$, an object of a category $\mathcal O$ in the same way, i.e.
\[
\mathcal O\simeq \mathrm{im}(e)\oplus \mathrm{im}(1-e).
\]
The main problem is that the notion of an image of an arrow could not exist in an arbitrary category. The Karoubi envelope is an extension of a category such that for a given idempotent $e$ the notions $\mathrm{im}(e)$ makes sense. Therefore one can ``split'' a given object in the Karoubi envelope that could be indecomposable in the category itself.

\begin{defin}\label{defin-karoubi}
Let $\mathcal C$ be a category and let $e\colon\mathcal O\to\mathcal O$ be an idempotent in $\mathrm{Mor}(\mathcal C)$. The \emph{Karoubi envelope of $\mathcal C$}, denoted $\KAR(\mathcal C)$, is the following category.
\begin{itemize}
\item Objects are ordered pairs $(\mathcal O,e)$ of an object $\mathcal O$ and an idempotent $e$ of $\mathcal C$.
\item Morphisms $f\colon(\mathcal O,e)\to(\mathcal O',e')$ are all arrows $f\colon\mathcal O\to\mathcal O'$ of $\mathcal C$ such that the equation $f=f\circ e=e'\circ f$ holds.
\item Compositions are defined in the obvious sense. The identity of an object is $e$ itself.
\end{itemize}
It is straightforward to check that this is indeed a category. We denote an object $(\mathcal O,e)$ as $\mathrm{im}(e)$, the \emph{image} of the idempotent $e$. Moreover, we identify the objects of $\mathcal C$ with their image via the embedding functor
\[
\iota\colon\mathcal C\to\KAR(\mathcal C),\;\mathcal O\mapsto (\mathcal O,\Id).
\] 
\end{defin}
Note that if $\mathcal C$ is pre-additive, then $1-e$ is also an projection and, under the identification above, we can finally write
\[
\mathcal O\simeq \mathrm{im}(e)\oplus \mathrm{im}(1-e).
\]
The following proposition is well-known (see e.g. \cite{bnsm}). The propositions allows us to shift the problem if two chain complexes are homotopy equivalent to the Karoubi envelope. Recall that $\kom(\mathcal C)$ denotes the category of formal chain complexes.
\begin{satz}\label{prop-karoubi}
Let $(C,d),(C',d')$ be two objects, i.e. formal chain complexes, of $\kom(\mathcal C)$. If the two objects are homotopy equivalent in $\kom(\KAR(\mathcal C))$, then the two objects are also homotopy equivalent in $\kom(\mathcal C)$.\qed
\end{satz}

We define the two orthogonal idempotents $\mathrm{u},\mathrm{d}$ next and show some basic, but very important, properties.

We note that we call the idempotents \emph{``down and up''}. The reader should be careful not to confuse them with the orientations on the resolutions or the colourings of Theorem \ref{theo-nonalternating}, i.e. latter colours change at v-crossings but ``down and up'' do not change. Note that (e) is very important in the following.

\begin{defin}\label{defin-idempotent}
We call the two cobordism in Figure \ref{figure4-3} the \emph{``down and up'' idempotents}. We denote them as $\mathrm{d}$ and $\mathrm{u}$.
\begin{figure}[ht]
  \centering
     \includegraphics[width=0.5\linewidth]{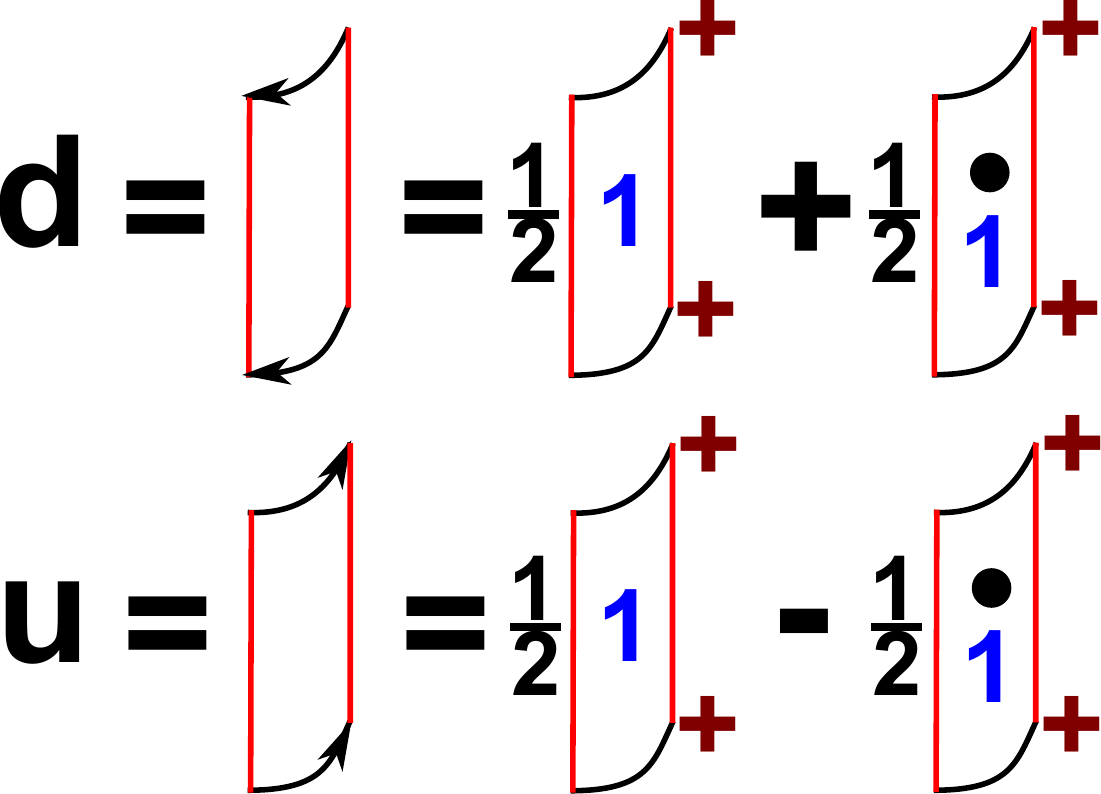}
  \caption{The two idempotents up and down.}
  \label{figure4-3}
\end{figure}
\end{defin}

\begin{lemm}\label{lemm-idem}
The cobordisms $\mathrm{d},\mathrm{u}$ satisfy the following identities.
\begin{enumerate}
\item[(a)] $\mathrm{d}^2=\mathrm{d}$ and $\mathrm{u}^2=\mathrm{u}$ (idempotent).
\item[(b)] $\mathrm{d}\circ\mathrm{u}=0=\mathrm{u}\circ\mathrm{d}$ (orthogonal).
\item[(c)] $\mathrm{d}+\mathrm{u}=\Id$ (complete).
\item[(d)] $\Id_{\mathrm{dot}}\circ\mathrm{d}=\mathrm{d}$ and $\Id_{\mathrm{dot}}\circ\mathrm{u}=-\mathrm{u}$ (Eigenvalues).
\item[(e)] $\Phi^-_+\circ\mathrm{d}=\mathrm{u}\circ\Phi^-_+$ and $\mathrm{d}\circ\Phi^-_+=\Phi^-_+\circ\mathrm{u}$ (change of orientations).
\item[(f)] $[\mathrm{d},\Phi^-_+]=\Id(1)_{\mathrm{dot}}=-[\mathrm{u},\Phi^-_+]$ (Commutator relation).
\end{enumerate}
\end{lemm}

\begin{proof}
All equations are straightforward to prove. One has to use the dot-relations from Figure \ref{figure3-1} and the relations from Definition \ref{defin-category}.

In (d)+(f) the surface $\Id(1)_{\mathrm{dot}}$ denotes an identity with an extra dot and $+1$ as an indicator.

Beware that the dot represents an handle. This forces a sign change after composition with the cobordism $\Phi^-_+$. The reader should compare this with the relations in Definition \ref{defin-category}.
\end{proof}

Now we take a look at the Karoubi envelope $\kom(\KAR(\ucob_R(k)))$. The discussion above shows that there is an isomorphism{\huge\[
\text{\none}\simeq\text{\down}\oplus\text{\up}.
\]}
With this notation we get
{\huge\[
\text{\smo}\simeq\text{\ddp}\oplus\text{\dup}\oplus\text{\udp}\oplus\text{\uup}
\]}
and analogous
{\huge\[
\text{\hsmo}\simeq\text{\ddpp}\oplus\text{\dupp}\oplus\text{\udpp}\oplus\text{\uupp}.
\]}
Recall that the standard orientation for the complex $\bn{\overcrossing}$ is (see e.g. Figure \ref{figure0-main})
\[
\bn{\overcrossing}=\text{\du}\xrightarrow{S(1)^{++}_{++}}\text{\ril}.
\]
To avoid mixing the notions of the down and up-colours and the orientations we denote this complex simply as $\bn{\overcrossing}^{++}_{++}$, i.e. standard orientations for all strings. Moreover, under the convention left=first top subscript, right=second top subscript, bottom=first bottom subscript and top=second bottom subscript, a notation like $\bn{\overcrossing}^{+-}_{-+}$ makes sense.

The following theorem is a main observation of this section.
\begin{theo}\label{theo-karoubi}
Let $(C,d)\simeq(C',d')$ denote two homotopy equivalent complexes. Then, in $\ukob_R(k)$, there are sixteen chain homotopies (only four are illustrated, but it should be clear how the rest look like)
\[
\begin{xy}
  \xymatrix{
      \bn{\overcrossing}^{++}_{++}\simeq \text{\raisebox{-0.1cm}{\dup}}\oplus\text{\raisebox{-0.1cm}{\udp}}\xrightarrow{0}\text{\raisebox{-0.1cm}{\dupp}}\oplus\text{\raisebox{-0.1cm}{\udpp}}, & \bn{\overcrossing}^{+-}_{-+}\simeq \text{\raisebox{-0.1cm}{\ddp}}\oplus\text{\raisebox{-0.1cm}{\uup}}\xrightarrow{0}\text{\raisebox{-0.1cm}{\ddpp}}\oplus\text{\raisebox{-0.1cm}{\uupp}},\\
      \bn{\overcrossing}^{++}_{+-}\simeq \text{\raisebox{-0.1cm}{\dup}}\oplus\text{\raisebox{-0.1cm}{\udp}}\xrightarrow{0}\text{\raisebox{-0.1cm}{\ddpp}}\oplus\text{\raisebox{-0.1cm}{\uupp}}, & \bn{\overcrossing}^{+-}_{++}\simeq \text{\raisebox{-0.1cm}{\ddp}}\oplus\text{\raisebox{-0.1cm}{\uup}}\xrightarrow{0}\text{\raisebox{-0.1cm}{\dupp}}\oplus\text{\raisebox{-0.1cm}{\udpp}}.
      }
\end{xy}
\]
Moreover, similar formulas hold for $\bn{\undercrossing}$. 
\end{theo}
\begin{proof}
We use the observations from above, i.e. in $\kom(\KAR(\ucob_R(k)))$ the differential of $\bn{\overcrossing}^{++}_{++}$ is a $4\times 4$-matrix of saddles. Hence, for $\bn{\overcrossing}^{++}_{++}$ we get (for simplicity write $S=S(1)^{++}_{++}$ and $S_{\mathrm{d}}$ and $S_{\mathrm{u}}$ for the saddle under the action of down and up)
\[
\begin{xy}
  \xymatrix{
       \text{\raisebox{-0.1cm}{\ddp}}\oplus\text{\raisebox{-0.1cm}{\udp}}\oplus\text{\raisebox{-0.1cm}{\dup}}\oplus\text{\raisebox{-0.1cm}{\uup}}\ar[rrr]^{\begin{pmatrix}
      S_{\mathrm{d}} & 0 & 0 & 0\\
      0 & 0 & 0 & 0\\
      0 & 0 & 0 & 0\\
      0 & 0 & 0 & S_{\mathrm{u}}
      \end{pmatrix}} & & & \text{\raisebox{-0.1cm}{\ddpp}}\oplus\text{\raisebox{-0.1cm}{\udpp}}\oplus\text{\raisebox{-0.1cm}{\dupp}}\oplus\text{\raisebox{-0.1cm}{\uupp}}.
      }
\end{xy}
\]
This is true, because all other saddles are killed by the orthogonality relations of the colours down and up.

Note that both non zero saddles are invertible, i.e. their inverses are the saddles
\[
\frac{1}{2}(S\colon \hsmoothing \to \smoothing)_{\mathrm{d}}\;\;\text{and}\;\;-\frac{1}{2}(S\colon \hsmoothing \to \smoothing)_{\mathrm{u}}
\]
with only $+$ as boundary decorations. To see this one uses Lemma \ref{lemm-idem} and the neck cutting relation. Thus we get
\[
\bn{\overcrossing}^{++}_{++}\simeq \text{\raisebox{-0.1cm}{\dup}}\oplus\text{\raisebox{-0.1cm}{\udp}}\xrightarrow{0}\text{\raisebox{-0.1cm}{\dupp}}\oplus\text{\raisebox{-0.1cm}{\udpp}}.
\]
To prove the further statements one has to use the relation (e) of Lemma \ref{lemm-idem}, i.e. the only surviving saddles change according to the action of $\Phi^-_+$. We note that this is a very important observation.

For $\bn{\undercrossing}^{++}_{++}$ one can simply copy the arguments from before.
\end{proof}
\begin{bem}\label{bem-importantsaddles}
We observe a very important fact. The four complexes similar to
\[
\bn{\overcrossing}=\text{\uu}\xrightarrow{S(1)^{+-}_{-+}}\text{\rir},
\] 
i.e. the ones for which we choose a non-alternating orientation of the resolutions, are the only ones that, due to the orthogonality relation for down and up, survive a closing with an even number of v-crossings, i.e. like $1,-1$ in Figure \ref{figure0-4}, in their source and target. So if we only consider v-link diagrams, then any such saddle is a multiplication, i.e. one bottom component, or a comultiplication, i.e. one top component. Therefore such morphisms will always be killed at the bottom (multiplication) or at the top (comultiplication).

And the only saddle that survives the closure of type $0$ in Figure \ref{figure0-4} is the M\"obius cobordism. Latter one is a $0$-morphism if $2$ is invertible (compare to the relations in Definition \ref{defin-category}). Hence, the dot-calculus of Definition \ref{defin-dotcalculus} ensures that all components will survive and the complex will be of the form
\[
\text{\raisebox{-0.1cm}{\ddp}}\oplus\text{\raisebox{-0.1cm}{\uup}}\xrightarrow{0}\text{\raisebox{-0.1cm}{\ddpp}}\oplus\text{\raisebox{-0.1cm}{\uupp}}.
\]
\end{bem}

As an application of the Theorem \ref{theo-semiloc} in Section \ref{circuit} and the Theorems \ref{theo-nonalternating} and \ref{theo-karoubi} above, we get the desired statement for v-link diagrams.
\begin{theo}[Degeneration]\label{theo-leedeg}
Let $L_D$ denote a $n$-component v-link diagram. Then $\bn{L_D}_{\mathrm{Lee}}$ is homotopy equivalent (in $\ukob_R(\emptyset)$) to a chain complex with only zero differentials and $2^n$ generators given by the $2^n$ non-alternating resolutions.

If $n=1$, i.e. $L_D$ is a v-knot diagram, then the two generators are in homology degree $0$.
\end{theo}
\begin{proof}
We will repress the notion of the formal signs of the morphisms to maintain readability. Moreover, we will choose a specific orientation for the resolutions. We can do both because of Lemma \ref{lemma-everythingfine}.

So the main part of the proof will be to choose the orientations in a good way and use Theorem \ref{theo-semiloc} (the reader may think of tensors). Recall that the number of crossings is finite. Hence, we can choose an orientation of any resolution such that the number $m$ of alternating crossings is minimal.

We observe that Theorem \ref{theo-semiloc} can be used because every crossing itself is a nice v-tangle diagram. Moreover, with Theorem \ref{theo-karoubi}, we see that the complex will be homotopy equivalent to a complex with only $0$-differentials. Hence, the only remaining thing is to show that the number of generators will work out as claimed.

Note that, if a resolution contains a lower part of a multiplication or a upper part of a comultiplication, then by Remark \ref{bem-importantsaddles} these resolution is killed. Moreover, we can ignore top and bottom parts of $\theta$, since they will always be non-alternating.

Now we define the \emph{dual graph of an resolution}, denoted $\mathcal D$, as follows. Recall that a resolution is a four valent graph without any c-crossings. Any edges of this graph is a vertex of $\mathcal D$. Two vertices are connected with a labeled edge iff they are connected by a v-crossing or a top part of a resolution of a multiplication or a bottom part of a comultiplication. First edges should be labeled $v$, the second type of edges should get an labeling thats corresponds to the given orientation of the resolution. We will work with the simple graph of that type, i.e. remove circles or parallel edges of the same type. See Figure \ref{figure4-6}, i.e. the figure shows two resolutions from Figure \ref{figure4-2} and their dual graphs. 
\begin{figure}[ht]
  \centering
     \includegraphics[width=0.8\linewidth]{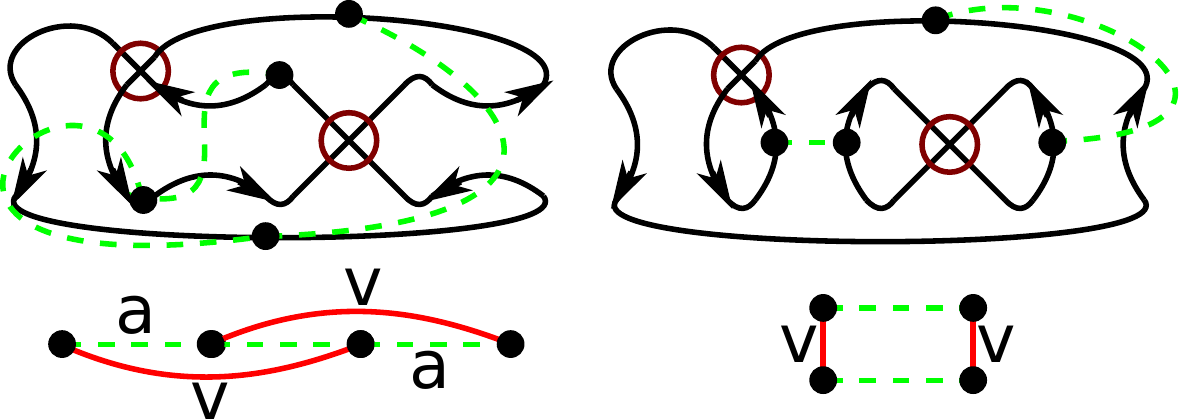}
  \caption{Two resolutions and their dual graphs.}
  \label{figure4-6}
\end{figure}

The advantage of this notation is that the question of surviving resolutions simplifies to the question of a colouring of the dual graph, i.e. a colouring of the dual graph is a colouring with two colours, say red and green, such that every $v$-labeled edge has two equally coloured adjacent vertices, every alternating crossing has also two different coloured adjacent vertices, but every non-alternating crossing has also two equal colours at adjacent vertices.

The reader should compare this to Theorem \ref{theo-karoubi} and Remark \ref{bem-importantsaddles}.

Then, because of Theorem \ref{theo-karoubi}, a resolution will have surviving generators iff it does not contain lower parts of multiplication or upper parts of comultiplications and, given an orientation of the resolution, it allows such a colouring.

For example, the left resolution in Figure \ref{figure4-6} do not allow such an colouring, but the right does.

The rest is just a case-by-case check, i.e. we have the following three cases. We use induction over $m$.
\begin{itemize}
\item[(i)] The dual graph of the resolution is a tree, i.e. no circles.
\item[(ii)] All circles in the dual graph have an even number of v-labeled edges.
\item[(iii)] There is one circle in the dual graph with an odd number of v-labeled edges.
\end{itemize}
If $m=0$, i.e. the resolution is non-alternating, we get exactly the claimed number of generators, since there are, by construction, no lower parts of multiplication or upper parts of comultiplications and the dual graph is of type (i) or (ii) and in both cases the graph can be coloured.

So let $m>0$ and let $C$ be an alternating crossing in a resolution $R$. As in Remark \ref{bem-importantsaddles} the whole resolution is killed if the $C$ is a lower part of a multiplication or an upper part of a comultiplication. Hence, we can assume that all alternating crossings of $R$ are either top components of multiplications or bottom components of comultiplications.

So we only have to check the three cases from above. If the resolution is one of type (i), then it is possible to choose the orientations in such a way that all crossings are non-alternating, i.e. this would be a contradiction to the minimality of $m$.

If the resolution is of type (ii), then the resolution only survives, i.e. the dual graph allows a colouring, iff the number of other alternating crossings in every circle is even. But in this case one can also choose an orientation with only non-alternating crossings. Hence, we would get a contradiction to the minimality of $m$ again.

An analogous argument works in the case of type (iii), i.e. the only surviving resolutions will also allow an non-alternating resolution.   

Hence, only non-alternating resolutions generate non vanishing objects. Moreover, any non-alternating resolution will create exactly two of these generators. Thus, with Theorem \ref{theo-nonalternating} the statement follows.
\end{proof}

Now we use the functor $\mathcal F_{\mathrm{Lee}}$ to get the corresponding statement in the category $\RMOD$. The reader may compare this to the classical results (e.g. see \cite{mtv}).
\begin{satz}\label{prop-algebraicdegen}
Let $L_D$ denote a $n$-component v-link diagram. Then we have the following.
\begin{itemize}
\item[(a)] If $R=\zet$, then their is an isomorphism
\[
H(\mathcal F_{\mathrm{Lee}}(L_D),R)\iso\bigoplus_{2^n}\zet\oplus \mathrm{Tor},
\]
where $\mathrm{Tor}$ is all torsion.
\item[(b)] If $R=\rat$ or $R=\zet\left[\frac{1}{2}\right]$, then their is an isomorphism
\[
H(\mathcal F_{\mathrm{Lee}}(L_D),R)\iso\bigoplus_{2^n}R.
\]
\item[(c)] The only possible torsion is $2$-torsion.
\end{itemize}
\end{satz}
\begin{proof}
The statement (b) follows from Theorem \ref{theo-leedeg} above. Recall that the whole construction requires that $2$ is invertible.

For (a)+(c) recall the universal coefficients theorem. i.e. their is a short exact sequence
\[
0\rightarrow H_*(\mathcal F_{\mathrm{Lee}}(L_D),\zet)\otimes_{\zet} R\rightarrow H_*(\mathcal F_{\mathrm{Lee}}(L_D),R)\rightarrow \mathrm{Tor}(H_{*+1}(\mathcal F_{\mathrm{Lee}}(L_D),\zet)),R)\rightarrow 0.
\]
Therefore (a) follows from (b) with $R=\rat$, since the Tor-functor will vanish in this case. And (c) follows from (b) with $R=\zet\left[\frac{1}{2}\right]$, since all non 2-torsion elements will exits for $\zet$ and $\zet\left[\frac{1}{2}\right]$.
\end{proof}

%Literatur

\vspace{0.1in}

\noindent D.T.: { \sl \small Courant Research Center ``Higher Order Structures'', University of G\"{o}ttingen, G\"{o}ttingen, Germany; 
Mathematisches Institut, Georg-August-Universit\"{a}t G\"{o}ttingen, G\"{o}ttingen, Germany} 
\newline \noindent {\tt \small email: dtubben@uni-math.gwdg.de}
\end{document}